\documentclass{article}

     \usepackage[nonatbib, final]{neurips_2019}

\usepackage[utf8]{inputenc} 
\usepackage[T1]{fontenc}    

\usepackage{url}            
\usepackage{booktabs}       
\usepackage{amsfonts}       
\usepackage{nicefrac}       
\usepackage{microtype}      
\usepackage{algorithm}
\usepackage{algorithmic}
\usepackage{color}
\usepackage{amssymb}
\usepackage{amsmath}
\usepackage{multirow}
\usepackage{makecell}
\usepackage{footnote}
\usepackage{graphicx}
\usepackage{xr}

\usepackage{hyperref}       
\def\algone{{\sf Varag }}
\def\newalgone{{\sf Varag}}
\def\eqnok#1{(\ref{#1})}
\def\argmin{{\mathrm{argmin}}}

\def\Argmin{{\mathrm{Argmin}}}
\def\exp{{\mathrm{exp}}}
\def\vgap{\vspace*{.1in}}

\providecommand{\dnorm}[1]{\|#1\|_*}

\newcommand{\beqa}{\begin{eqnarray}}
\newcommand{\eeqa}{\end{eqnarray}}
\newcommand{\beqas}{\begin{eqnarray*}}
\newcommand{\eeqas}{\end{eqnarray*}}
\newcommand{\eat}[1]{}

\newcommand{\E}{{\mathbb{E}}}

\newcommand{\tsum}{\textstyle{\sum}}

\newcommand{\beq}{\begin{equation}}
\newcommand{\eeq}{\end{equation}}
\newcommand{\nn}{\nonumber}
\newcommand{\bbe}{\mathbb{E}}
\newcommand{\dnorms}[1]{\|#1\|_*^2}
\newcommand{\bbr}{\mathbb{R}}

\newenvironment{proof}{\noindent {\em Proof: }\ignorespaces}%
                {\hspace*{\fill}$\Box$\par}
        {\hspace*{\fill}$\Box$\par\vspace{4mm}}
\newenvironment{proofof}[1]{\smallskip\noindent{\bf Proof of #1.}}%
        {\hspace*{\fill}$\Box$\par}

\newtheorem{theorem}{Theorem}
\newtheorem{lemma}{Lemma}

\newtheorem{remark}{Remark}

\newcommand{\todo}[1]{#1}

\title{A unified variance-reduced accelerated gradient method for convex optimization}

\author{%
  Guanghui Lan\\
  H. Milton Stewart School of Industrial \& Systems Engineering\\
   Georgia Institute of Technology\\
    Atlanta, GA 30332\\
    \texttt{george.lan@isye.gatech.edu}\\
    \And
    Zhize Li\\
    Institute for Interdisciplinary Information Sciences\\
    Tsinghua University\\
    Beijing 100084, China\\
    \texttt{zz-li14@mails.tsinghua.edu.cn}
    \And
    Yi Zhou\\
    IBM Almaden Research Center\\
    San Jose, CA 95120\\
    \texttt{yi.zhou@ibm.com}
}

\begin{document}

\maketitle

\begin{abstract}
We propose a novel randomized incremental gradient algorithm, namely, VAriance-Reduced Accelerated Gradient (\newalgone),
for finite-sum optimization.
Equipped with a unified step-size policy that adjusts itself to the value of the condition number, 
\algone exhibits the unified optimal rates of convergence for solving smooth convex finite-sum problems directly regardless of their strong convexity.
Moreover, 
\algone is the first accelerated randomized incremental gradient method that benefits from the strong convexity of the data-fidelity term to achieve 
the optimal linear convergence.
It also establishes an optimal linear rate of convergence for solving 
a wide class of problems only satisfying a certain error bound condition rather than strong convexity.
\algone can also be extended to solve stochastic finite-sum problems.
\end{abstract}

\section{Introduction}\label{sec:intro}
The problem of interest in this paper is the convex programming (CP) problem given in the form of
\beq\label{cp}
\psi^*:=\min_{x \in X} \left\{\psi(x):=\tfrac{1}{m}\tsum_{i=1}^{m}f_i(x)+ h(x)\right\}.
\eeq
Here, $X\subseteq \bbr^n$ is a closed convex set, the component function
$f_i:X\rightarrow\bbr,\ \ i=1,\dots,m,$ are smooth and convex function with $L_i$-Lipschitz continuous gradients over $X$, i.e.,
$\exists L_i\ge 0$ such that
\beq\label{def_smoothness}
\|\nabla f_i(x_1)-\nabla f_i(x_2)\|_*\le L_i \|x_1-x_2\|, \ \ \forall x_1,x_2 \in X,
\eeq
and $h: X\rightarrow\bbr$ is a relatively simple but possibly nonsmooth convex function.
For notational convenience, we denote
$f(x) := \tfrac{1}{m}\tsum_{i=1}^m f_i(x)$ and
$L := \tfrac{1}{m}\tsum_{i=1}^m L_i$.
It is easy to see that $f$ has $L$-Lipschitz continuous gradients, i.e., for some $L_f \ge 0$,
$
\|\nabla f(x_1) - \nabla f(x_2) \|_* \le L_f \|x_1 - x_2\| \le L \|x_1- x_2\|, \ \ \forall x_1, x_2 \in X.
$
It should be pointed out that it is not necessary to assume $h$ being strongly convex. Instead, we assume that
$f$ is possibly strongly convex with modulus $\mu \ge 0$.

\todo{We also consider a class of stochastic finite-sum optimization problems given by
\beq\label{sp}
\psi^*:=\min_{x \in X} \left\{\psi(x):=\tfrac{1}{m}\tsum_{i=1}^{m}\bbe_{\xi_i}[F_i(x,\xi_i)]+ h(x)\right\},
\eeq
where $\xi_i$'s are random variables with support $\Xi_i\subseteq \bbr^d$. It can be easily seen that \eqref{sp} is a special case of \eqref{cp} with $f_i=\bbe_{\xi_i}[F_i(x,\xi_i)], i=1,\dots, m$. However, different from deterministic finite-sum optimization problems, only noisy gradient information of each component function $f_i$ can be accessed for the stochastic finite-sum optimization problem in \eqref{sp}.
Particularly, \eqref{sp} models the generalization risk minimization in distributed machine learning problems.}


Finite-sum optimization given in the form of \todo{\eqref{cp} or \eqref{sp}} has recently found a wide range of applications in machine learning (ML), statistical inference, and image processing, and hence becomes the subject of intensive studies during the past few years.
In centralized ML, $f_i$ usually denotes the loss generated by a single data point, while in distributed ML, it 
may correspond to the loss function for an agent $i$
, which is connected to other agents 
in a distributed network.

Recently,
randomized incremental gradient (RIG) methods have emerged as an important class of first-order methods for finite-sum optimization (e.g.,\cite{BlHeGa07-1,JohnsonZhang13-1,xiao2014proximal,DefBacLac14-1,SchRouBac13-1,lan2015optimal,allen2016katyusha,allen2016improved,hazan2016variance,LinMaiHar15-1,lan2018random}).
In an important work, \cite{SchRouBac13-1} (see \cite{BlHeGa07-1} for a precursor) showed that by incorporating new gradient estimators into stochastic gradient descent (SGD) one can possibly
achieve a linear rate of convergence for smooth and strongly convex finite-sum optimization.
Inspired by this work, \cite{JohnsonZhang13-1} proposed a stochastic variance reduced gradient (SVRG) 
which incorporates a novel stochastic estimator of $\nabla f(x_{t-1})$.
More specifically, each epoch of SVRG starts with the computation of the exact gradient
 $\tilde g = \nabla f(\tilde x)$ for a given $\tilde x\in \bbr^n$ and then runs SGD for a fixed number of steps
using the gradient estimator
 \[
 G_t = (\nabla f_{i_t}(x_{t-1}) - \nabla f_{i_t}(\tilde x)) + \tilde g,
 \]
 where $i_t$ is a random variable with support on $\{1,\dots,m\}$.
 They show that the variance of $G_t$ vanishes as the algorithm proceeds,
 and hence SVRG exhibits an improved linear rate of convergence, i.e., ${\cal O}\{(m+L/\mu)\log(1/\epsilon)\}$, for smooth and strongly convex finite-sum problems.
See \cite{xiao2014proximal,DefBacLac14-1} for the same complexity result.
Moreover, \cite{allen2016improved} show that by doubling the epoch length 
SVRG obtains an ${\cal O}\{m\log(1/\epsilon)+L/\epsilon\}$ complexity bound for smooth convex finite-sum optimization.

Observe that the aforementioned variance reduction methods are not accelerated and hence they are not optimal even when the number of components $m=1$. 
Therefore, much recent research effort has been devoted to the design of optimal RIG methods. 
In fact, \cite{lan2015optimal} established a lower complexity bound for RIG methods 
by showing that whenever the dimension is large enough,
the number of gradient evaluations required by any RIG methods to find an $\epsilon$-solution of a smooth and strongly convex finite-sum problem 
i.e., a point $\bar x \in X$ s.t.
$\bbe[\|\bar x - x^*\|^2_2] \le \epsilon$, cannot be smaller than
\beq \label{RIG_lb}
{\Omega} \left( \left(m + \sqrt{\tfrac{m L}{\mu}}\right) \log \tfrac{1}{\epsilon}\right).
\eeq

As can be seen from Table~\ref{tab:summary}, existing accelerated RIG methods are optimal for solving smooth and strongly convex finite-sum problems, since their complexity matches the lower bound in \eqref{RIG_lb}.
%
%

Notwithstanding these recent progresses, there still remain a few significant issues on the development of accelerated RIG methods.
Firstly, 
as pointed out by \cite{tang2018rest}, existing RIG methods can only establish accelerated linear convergence based on the assumption that
  the regularizer $h$ is strongly convex, and fails to benefit from the strong convexity from the data-fidelity term \cite{wang2017exploiting}.
This restrictive assumption does not apply to many important applications (e.g., Lasso models) where the loss function, rather than the regularization term, may be strongly convex. 
Specifically, when dealing with the case that only $f$ is strongly convex but not $h$, 
one may not be able to shift the strong convexity of $f$\todo{, by subtracting and adding a strongly convex term,} to construct a simple strongly convex term $h$ in the objective function. 
In fact, even if $f$ is strongly convex, some of the component functions $f_i$ may only be convex, and hence \todo{these $f_i$s} may become nonconvex after subtracting a strongly convex term.
Secondly, if the strongly convex modulus $\mu$ becomes very small, the complexity bounds of all existing RIG methods will
go to $+\infty$ (see column 2 of Table~\ref{tab:summary}), indicating that they are not robust against problem ill-conditioning.
Thirdly, for solving smooth problems without strong convexity, one has to add a strongly convex perturbation into the objective function in order to 
gain up to a factor of $\sqrt{m}$ over Nesterov's accelerated gradient method for gradient computation (see column 3 of Table~\ref{tab:summary}).
One significant difficulty for this indirect approach is 
that we do not know how to choose the perturbation parameter properly, especially for problems with unbounded feasible region
(see \cite{allen2016improved} for a discussion about a similar issue related to
SVRG applied to non-strongly convex problems). However, if one chose not to add the strongly convex perturbation term,
the best-known complexity would be given by Katyusha\textsuperscript{ns}\cite{allen2016katyusha}, which are not more advantageous over Nesterov's orginal method.
In other words, it does not gain much from randomization in terms of computational complexity.
Finally, it should be pointed out that only a few existing RIG methods, e.g., RGEM\cite{lan2018random} and \cite{kulunchakov2019estimate}, can be applied to solve stochastic finite-sum optimization problems, where one can only access the stochastic gradient of $f_i$ via a stochastic first-order oracle (SFO). 

%

\begin{table}[ht]
    \centering
    \caption{\small Summary of the recent results on accelerated RIG methods}
    \label{tab:summary}
    \small
    \begin{tabular}{|c|c|c|}
      \hline
        Algorithms &  Deterministic smooth strongly convex & Deterministic smooth convex \\
        \hline
        RPDG\cite{lan2015optimal} & ${\cal O}\left\{(m+\sqrt{\tfrac{mL}{\mu}})\log\tfrac{1}{\epsilon}\right\}$ & ${\cal O}\left\{(m + \sqrt{\tfrac{mL}{\epsilon}})\log\tfrac{1}{\epsilon}\right\}$\footnotemark[1]\\
        \hline
        Catalyst\cite{LinMaiHar15-1} & ${\cal O}\left\{(m+\sqrt{\tfrac{mL}{\mu}})\log\tfrac{1}{\epsilon}\right\}$\footnotemark[1]& ${\cal O}\left\{(m+\sqrt{\tfrac{mL}{\epsilon}})\log^2\tfrac{1}{\epsilon}\right\}$\footnotemark[1]\\
        \hline
         Katyusha\cite{allen2016katyusha} & ${\cal O}\left\{(m+\sqrt{\tfrac{mL}{\mu}})\log\tfrac{1}{\epsilon}\right\}$ &  ${\cal O}\left\{(m\log\tfrac{1}{\epsilon}+\sqrt{\tfrac{mL}{\epsilon}})\right\}$\footnotemark[1]\\
         \hline 
         Katyusha\textsuperscript{ns}\cite{allen2016katyusha} & NA & ${\cal O}\left\{\tfrac{m}{\sqrt \epsilon} + \sqrt{\tfrac{mL}{\epsilon}}\right\}$\\
        \hline
        RGEM\cite{lan2018random} & ${\cal O}\left\{(m+\sqrt{\tfrac{mL}{\mu}})\log\tfrac{1}{\epsilon}\right\}$& NA\\
        \hline
    \end{tabular}
\end{table}
\footnotetext[1]{These complexity bounds are obtained via indirect approaches, i.e., by adding strongly convex perturbation.}

\noindent{\bf Our contributions.}
In this paper, we propose a novel accelerated variance reduction type method, 
namely the {\bf va}riance-{\bf r}educed {\bf a}ccelerated {\bf g}radient (\newalgone) method, 
to solve smooth finite-sum optimization problems given in the form of \eqref{cp}. Table~\ref{tab:summary-alg} summarizes the main convergence results achieved by our \algone algorithm.

\begin{table}[ht]
    \centering
    \caption{\small Summary of the main convergence results for \algone}
    \label{tab:summary-alg}
    \small
    \begin{tabular}{|c|c|c|}
        \hline
        Problem & Relations of $m$, $1/\epsilon$ and $L/\mu$ & Unified results\\
        \hline
        \multirow{3}{*}{\makecell{\\smooth optimization problems \eqref{cp} \\
        with or without strong convexity 
        }} &
        $m\ge \tfrac{D_0}{\epsilon}$ \footnotemark[2] or $m\ge \tfrac{3L}{4\mu} $& ${\cal O}\left\{m\log\tfrac{1}{\epsilon}\right\}$ \\
        \cline{2-3}
         &$m<\tfrac{D_0}{\epsilon}\le \tfrac{3L}{4\mu}$& ${\cal O}\left\{m\log m+\sqrt{\tfrac{mL}{\epsilon}}\right\}$ \\
        \cline{2-3}
        & $m<\tfrac{3L}{4\mu}\le \tfrac{D_0}{\epsilon}$ & ${\cal O}\left\{m\log m+\sqrt{\tfrac{mL}{\mu}}\log\tfrac{D_0/\epsilon}{3L/4\mu}\right\}$ \footnotemark[3] \\
        \hline
    \end{tabular}
\end{table}
\footnotetext[2]{$D_0 = 2[\psi(x^0) - \psi(x^*)] + 3LV(x^0, x^*)$ where $x^0$ is the initial point, $x^*$ is the optimal solution of \eqref{cp} and $V$ is defined in \eqref{primal_prox}.}

\footnotetext[3]{Note that this term is less than ${\cal O}\{\sqrt{\tfrac{mL}{\mu}}\log\tfrac{1}{\epsilon}\}$.}


Firstly, for smooth convex finite-sum optimization, our proposed method exploits a direct acceleration scheme 
instead of employing any perturbation or restarting techniques to obtain desired optimal convergence results.
As shown in the first two rows of Table~\ref{tab:summary-alg}, \algone achieves the optimal rate of convergence
if the number of component functions $m$ is relatively small and/or the required accuracy is high,
while it exhibits a fast linear rate of convergence 
when the number of component 
functions $m$ is relatively large and/or the required accuracy is low, without requiring any strong convexity assumptions. 
To the best of our knowledge, this is the first time that these complexity bounds 
have been obtained through a direct acceleration scheme for smooth convex finite-sum optimization in the literature.
In comparison with existing methods using perturbation techniques, \algone does not need to know
the target accuracy or the diameter of the feasible region a priori, and thus can be used 
to solve a much wider class of smooth convex problems, e.g., those with unbounded feasible sets.
%

Secondly, we equip \algone with a unified step-size policy for smooth convex optimization 
no matter \eqref{cp} is strongly convex or not, i.e., the strongly convex modulus $\mu \ge 0$.
With this step-size policy, \algone can adjust to different classes of problems to 
achieve the best convergence results, without knowing the target accuracy and/or 
fixing the number of epochs. 
In particular, as shown in the last column of Table~\ref{tab:summary-alg}, when $\mu$ is relatively large, \algone achieves the well-known optimal linear rate of convergence.
If $\mu$ is relatively small, e.g., $\mu<\epsilon$, it obtains the accelerated convergence rate 
that is independent of the condition number $L/\mu$. 
Therefore, \algone is robust against ill-conditioning of problem~\eqref{cp}.
Moreover, our assumptions on the objective function is more general comparing to those used by other RIG methods, such as RPDG and Katyusha.
Specifically, 
 \algone does not require to keep a strongly convex regularization term in the projection, and so we
can assume that the strong convexity is associated with the smooth function $f$ instead of the simple proximal function $h(\cdot)$.
Some other advantages of \algone over existing accelerated SVRG methods, e.g., Katyusha, include that it only requires
the solution of one, rather than two, subproblems, and that it can allow the application
of non-Euclidean Bregman distance for solving all different classes of problems.

Finally, we extend \algone to solve two more general classes of finite-sum optimization problems.
We demonstrate that \algone is the first randomized method that achieves the accelerated linear rate of convergence when solving the class of problems that satisfies a certain error-bound condition rather than strong convexity. 
We then show that \algone can also be applied to solve stochastic smooth finite-sum optimization problems resulting in a sublinear rate of convergence. 

This paper is organized as follows. In Section~\ref{sec:results}, we present our proposed algorithm \algone and its convergence results for solving \eqref{cp} under different problem settings. 
In Section~\ref{sec:numerical} we provide extensive experimental results to demonstrate the advantages of \algone over several state-of-the-art methods for solving some well-known ML models, e.g., logistic regression, Lasso, etc.
We defer the proofs of the main results in Appendix~\ref{sec:analysis_deter}.

\noindent{\bf Notation and terminology.}
We use $\|\cdot\|$ to denote a general norm in $\bbr^n$ without specific mention,
and $\|\cdot\|_*$ to denote the conjugate norm of $\|\cdot\|$.
For any $p\ge 1$, $\|\cdot\|_p$ denotes the standard $p$-norm in $\bbr^n$, i.e.,
$
\|x\|^p_p=\tsum_{i=1}^{n}|x_i|^p, \  \mbox{for any } x\in \bbr^n.
$
%
%
For a given strongly convex function $w:X\rightarrow \bbr$ with modulus $1$ w.r.t. an arbitrary norm $\|\cdot\|$, we define a {\sl prox-function} associated with $w$ as
\beq\label{primal_prox}
V(x^0,x)\equiv V_{w}(x^0,x):=w(x)-\left[w(x^0)+\langle w'(x^0),x-x^0\rangle\right],
\eeq
where $w'(x^0)\in \partial w(x^0)$ is any subgradient of $w$ at $x^0$.
By the strong convexity of $w$, we have
\beq\label{V_strong}
V(x^0,x)\ge \tfrac{1}{2}\|x-x^0\|^2, \ \ \forall x,x^0\in X.
\eeq
Notice that $V(\cdot,\cdot)$ described above is different from the standard definition for Bregman distance~\cite{Breg67,AuTe06-1,BBC03-1,Kiw97-1,censor1981iterative} in the sense that $w$ is not necessarily differentiable.
Throughout this paper, we assume that the prox-mapping associated with $X$ and $h$, given by
\beq\label{prox_mapping}
\argmin_{x\in X}\left\{\gamma[\langle g,x\rangle+ h(x)+ \mu V(\underline x_0,x)] + V(x_{0},x) \right\},
\eeq
can be easily computed for any $\underline x_0, x_0 \in X, g\in \bbr^{n}, \mu\ge 0, \gamma>0$.
We denote logarithm with base $2$ as $\log$. For any real number $r$, $\lceil r \rceil$ and $\lfloor r \rfloor$ denote the ceiling and floor of
$r$. 

\setcounter{equation}{0}

\section{Algorithms and main results}\label{sec:results}
This section contains two subsections.
We first present in Subsection~\ref{sec:alg-deterministic} a unified optimal \algone for solving the finite-sum problem given in \eqref{cp} as well as its optimal convergence results.
Subsection~\ref{sec:gen} is devoted to the discussion of several extensions of \newalgone.
Throughout this section, we assume that each component function $f_i$ is smooth with $L_i$-Lipschitz continuous gradients over $X$, i.e., \eqref{def_smoothness} holds for all component functions.
Moreover, we assume that the objective function $\psi(x)$ is possibly strongly convex, in particular, for $f(x)=\tfrac{1}{m}\tsum_{i=1}^{m}f_i(x)$, $\exists \mu\ge 0$ s.t.
\beq\label{def_strongconvexity}
f(y)\ge f(x)+\langle \nabla f(x), y-x\rangle + \mu V(x,y), \forall x, y\in X.
\eeq
Note that we assume the strong convexity of $\psi$ comes from $f$, and the simple function $h$ is not necessarily strongly convex.
Clearly the strong convexity of $h$, if any, can be shifted to $f$ since $h$ is assumed to be simple and its structural information is transparent to us. 
Also observe that \eqref{def_strongconvexity} is defined based on a generalized Bregman distance,
and together with \eqref{V_strong} they imply the standard definition of strong convexity w.r.t. Euclidean norm.

\subsection{\algone for convex finite-sum optimization}\label{sec:alg-deterministic}
The basic scheme of \algone is formally described in Algorithm~\ref{algVRASGD}.
In each epoch (or outer loop), it first computes the full gradient $\nabla f(\tilde x)$ at the point $\tilde x$ (cf. Line~\ref{eqn:fullgradient}), which 
will then be repeatedly used to define a gradient estimator $G_t$ at each iteration of the inner loop (cf. Line~\ref{eqn:estgradient}). This is the well-known variance reduction technique employed by many algorithms (e.g., \cite{JohnsonZhang13-1,xiao2014proximal,allen2016katyusha,hazan2016variance}).
The inner loop has a similar algorithmic scheme to the accelerated stochastic approximation algorithm \cite{Lan08-1,GhaLan12-2a,GhaLan13-1}
with a constant step-size policy. Indeed, the parameters used in the inner loop, i.e., $\{\gamma_s\}, \{\alpha_s\}$, and $\{p_s\}$, 
only depend on the index of epoch $s$. Each iteration of the inner loop requires the gradient information of only one randomly 
selected component function $f_{i_t}$, and maintains three primal sequences, $\{\underline x_t\}, \{x_t\}$ and $\{\bar x_t\}$, which play important role in the acceleration scheme.

\begin{algorithm}[h]
  \caption{The VAriance-Reduced Accelerated Gradient (\newalgone) method}
  \label{algVRASGD}
  \begin{algorithmic}[1]
\REQUIRE
 $x^0\in X, \{T_s\}, \{\gamma_s\}, \{\alpha_s\}, \{p_s\}, \{\theta_t\}$, and a probability distribution $Q=\{q_1,\ldots,q_m\}$ on $\{1,\ldots,m\}$.
\STATE Set $\tilde x^0 = x^0$.
\FOR {$s=1,2,\ldots$}
\STATE Set $\tilde{x}=\tilde{x}^{s-1}$ and $\tilde{g} = \nabla f(\tilde{x})$\label{eqn:fullgradient}.
\STATE Set $x_0=x^{s-1}$, $\bar x_0 = \tilde x$ and $T = T_s$.
\FOR {$t=1,2,\ldots,T$}
\STATE 
Pick $i_t\in \{1,\dots,m\}$ randomly according to $Q$. 
\STATE $\underline x_t = \left[(1+\mu \gamma_s)(1  - \alpha_s - p_s) \bar x_{t-1}  + \alpha_s x_{t-1}+ (1+\mu \gamma_s) p_s \tilde x\right]/ [1+\mu \gamma_s(1-\alpha_s)]$\label{eqn:underlinex}.
\STATE $G_t = (\nabla f_{i_t}(\underline x_{t}) - \nabla f_{i_t}(\tilde{x}))/(q_{i_t} m)  + \tilde{g}$\label{eqn:estgradient}.
\STATE $x_{t} = \arg \min_{x \in X} \left\{ \gamma_s \left[\langle G_t, x \rangle
+ h(x) + \mu V(\underline x_t, x)\right] + V(x_{t-1}, x) \right\}$\label{eqn:xt}.
\STATE $\bar x_{t} =   (1 - \alpha_s-p_s) \bar x_{t-1}  + \alpha_s x_{t} + p_s \tilde x$\label{eqn:barx}.
\ENDFOR
\STATE Set  $x^s = x_{T}$ and $\tilde{x}^s =\tsum_{t=1}^{T}(\theta_t \bar x_t)/\tsum_{t=1}^{T} \theta_t$.
\ENDFOR
  \end{algorithmic}
\end{algorithm}
Note that \algone is closely related to stochastic mirror descent method \cite{NJLS09-1,nemyud:83} and SVRG\cite{JohnsonZhang13-1,xiao2014proximal}. By setting $\alpha_s= 1$ and $p_s = 0$, Algorithm~\ref{algVRASGD} simply combines the variance reduction technique with stochastic mirror descent. In this case, the algorithm only maintains one primal sequence $\{x_t\}$ and possesses the non-accelerated rate of convergence ${\cal O}\{(m+L/\mu)\log(1/\epsilon)\}$ for solving \eqref{cp}.
Interestingly, if we use Euclidean distance instead of prox-function $V(\cdot,\cdot)$ to update $x_t$ and set $X = \bbr^n$,  Algorithm~\ref{algVRASGD} will further reduce to prox-SVRG proposed in \cite{xiao2014proximal}.

It is also interesting to observe the difference between \algone and Katyusha \cite{allen2016katyusha} because both are
accelerated variance reduction methods.
Firstly, while Katyusha needs to assume that the strongly convex term is specified as in the form of a simple proximal function, e.g., $\ell_1$/$\ell_2$-regularizer,
\algone assumes that $f$ is possibly strongly convex, 
which solves an open issue of the existing accelerated RIG methods pointed out by \cite{tang2018rest}.
Therefore, the momentum steps in Lines~\ref{eqn:underlinex} and \ref{eqn:barx} are different from Katyusha.
Secondly, \algone has a less computationally expensive algorithmic scheme. 
Particularly, \algone only needs to solve one proximal mapping (cf. Line~\ref{eqn:xt}) per iteration even if $f$ is strongly convex, while Katyusha requires to solve two proximal mappings per iteration.
Thirdly, \algone incorporates a prox-function $V$ defined in \eqref{primal_prox} rather than the Euclidean distance in the proximal mapping to update $x_t$. 
This allows the algorithm to take advantage of the geometry of the constraint set $X$ when performing projections. 
However, Katyusha cannot be fully adapted to the non-Euclidean setting because its second proximal mapping 
must be defined using the Euclidean distance regardless the strong convexity of $\psi$.
Finally, we will show in this section that \algone can achieve a much better rate of convergence than Katyusha for smooth convex finite-sum optimization by using a novel approach to specify step-size and to schedule epoch length.

We first discuss the case when $f$ is not necessarily strongly convex, i.e., $\mu = 0$ in \eqref{def_strongconvexity}.
In Theorem~\ref{Them:main-deter-smooth}, we suggest one way to specify the algorithmic parameters, including $\{q_i\}$, $\{\theta_t\}$, $\{\alpha_s\}$, $\{\gamma_s\}$, $\{p_s\}$ and $\{T_s\}$,
for \algone to solve smooth convex problems given in the form of \eqref{cp}, and discuss its convergence properties of the resulting algorithm. 
We defer the proof of this result in Appendix~\ref{sec:deter-sm-pf}.

\begin{theorem}[\bf Smooth finite-sum optimization]\label{Them:main-deter-smooth}
Suppose that the probabilities $q_i$'s are set to $L_i / \tsum_{i=1}^m L_i$ for $i = 1, \ldots, m$, and weights $\{\theta_t\}$ are set as
\beq \label{def_theta_acc_SVRG}
\theta_t =
\begin{cases}
\tfrac{\gamma_{s}}{\alpha_{s}} (\alpha_{s} + p_{s}) & 1 \le t \le T_s-1\\
 \tfrac{\gamma_s}{\alpha_s} & t=T_s.
 \end{cases}
\eeq
Moreover, let us denote $s_0 := \lfloor \log m\rfloor+1$ and set parameters $\{T_s\}$, $\{\gamma_s\}$ and $\{p_s\}$ as
\begin{align}\label{parameter-deter-smooth1}
T_s = \begin{cases}
2^{s-1}, & s \le s_0\\
T_{s_0}, & s > s_0
\end{cases}, \
\gamma_s = \tfrac{1}{3 L \alpha_s}, \
\ \mbox{and} \
p_s = \tfrac{1}{2}, \ \mbox{with}
\end{align}
\begin{align}\label{parameter-deter-alpha-sm}
    \alpha_s =
\begin{cases}
\tfrac{1}{2}, & s \le s_0\\
\tfrac{2}{s-s_0+4},& s > s_0
\end{cases}.
\end{align}
Then the total number of gradient evaluations of $f_i$ performed by Algorithm~\ref{algVRASGD} to find a stochastic $\epsilon$-solution of \eqref{cp}, i.e., a point $\bar x\in X$ s.t. $\bbe[\psi(\bar x)-\psi^*]\le \epsilon$,  can be bounded by
\beq\label{complexity-deter-sm}
\bar N := \begin{cases}
{\cal O}\left\{  m \log \tfrac{D_0}{\epsilon} \right\}, & m \ge D_0/\epsilon,\\
 {\cal O} \left\{m \log m + \sqrt{\tfrac{m D_0}{\epsilon}} \right\}, &  m < D_0/\epsilon,
\end{cases}
\eeq
where $D_0$ is defined as
\beq \label{def_D_0}
D_0:= 2[\psi(x^0) - \psi(x^*)] + 3L V(x^0, x^*).
\eeq
\end{theorem}

We now make a few observations regarding the results obtained in Theorem~\ref{Them:main-deter-smooth}.
Firstly, as mentioned earlier, whenever the required accuracy $\epsilon$ is low and/or the number of components $m$ is large, \algone can achieve a fast linear rate of convergence even under the assumption that the objective function is not strongly convex. 
Otherwise, \algone achieves an optimal sublinear rate of convergence with complexity bounded by ${\cal O} \{\sqrt{m D_0/\epsilon} + m \log m \}$. 
Secondly, whenever $\sqrt{m D_0/\epsilon}$ is dominating in the second case of \eqref{complexity-deter-sm}, \algone can save up to ${\cal O}(\sqrt m)$ 
gradient evaluations of $f_i$ than the optimal deterministic first-order methods for solving \eqref{cp}. To the best our knowledge,
\algone is the first accelerated RIG in the literature to obtain such convergence results by directly solving \eqref{cp}.
Other existing accelerated RIG methods, such as RPDG\cite{lan2015optimal} and Katyusha\cite{allen2016katyusha}, require the application of
perturbation and restarting techniques to obtain such convergence results.
\todo{Thirdly, \algone also supports mini-batch approach where the component function $f_i$ is associated with a mini-batch of data samples instead of a single data sample. 
In a more general case, for a given mini-batch size $b$, we assume that the component functions can be split into subsets where each subset contains exactly $b$ number of component functions. Therefore, one can replace Line~\ref{eqn:estgradient} in Algorithm~\ref{algVRASGD} by $G_t = \tfrac{1}{b}\tsum_{i_t\in S_b} (\nabla f_{i_t}(\underline x_t) - \nabla f_{i_t}(\tilde x))/(q_{i_t}m) + \tilde g$ with $S_b$ being the selected subset and $|S_b|=b$ and adjust the appropriate parameters to obtain the mini-batch version of \newalgone.
The mini-batch \algone can obtain parallel linear speedup of factor $b$ whenever the mini-batch size $b\le \sqrt{m}$.}

Next we consider the case when $f$ is possibly strongly convex, including the situation
when the problem is almost not strongly convex, i.e., $\mu \approx 0$.
In the latter case, the term $\sqrt{mL/\mu}\log(1/\epsilon)$ will be dominating in the complexity of existing accelerated RIG methods (e.g., \cite{lan2015optimal,lan2018random,allen2016katyusha,LinMaiHar15-1})
and will tend to $\infty$ as $\mu$ decreases. Therefore, these complexity bounds are significantly worse than \eqref{complexity-deter-sm} obtained by simply treating \eqref{cp} as smooth convex problems.
Moreover, $\mu\approx 0$ is very common in ML applications.
In Theorem~\ref{Them:main-deter-sc}, we provide a unified step-size policy which allows \algone to achieve optimal rate of convergence for finite-sum optimization in \eqref{cp} regardless of its strong convexity, and hence it can achieve stronger rate of convergence than existing accelerated RIG methods if the condition number $L/\mu$ is very large.
The proof of this result can be found in Appendix~\ref{sec:deter-unified-pf}.

\begin{theorem}[\bf A unified result for convex finite-sum optimization]\label{Them:main-deter-sc}
Suppose that the probabilities $q_i$'s are set to $L_i / \tsum_{i=1}^m L_i$ for $i = 1, \ldots, m$.
Moreover, let us denote $s_0 := \lfloor \log m\rfloor+1$ and assume that the
weights $\{\theta_t\}$ are set to \eqref{def_theta_acc_SVRG} if $1\le s \le s_0$ or $s_0< s \le s_0+\sqrt{\tfrac{12L}{m \mu}}-4, \ m < \tfrac{3L}{4\mu}$.
Otherwise, they are set to
\beq \label{def_theta_acc_SVRG_sc}
\theta_t =
\begin{cases}
\Gamma_{t-1} - (1 - \alpha_s - p_s) \Gamma_{t}, & 1 \le t \le T_s-1,\\
\Gamma_{t-1}, & t = T_s,
\end{cases}
\eeq
where $\Gamma_t= (1+\mu\gamma_s)^t$.
If the parameters $\{T_s\}$, $\{\gamma_s\}$ and $\{p_s\}$ set  to \eqref{parameter-deter-smooth1} with 
\begin{align}\label{parameter-deter-alpha-unified}
\alpha_s =
\begin{cases}
\tfrac{1}{2}, & s \le s_0,\\
\max\left\{\tfrac{2}{s-s_0+4}, \min\{\sqrt{\tfrac{m \mu}{3L}}, \tfrac{1}{2}\}\right\},& s > s_0,
\end{cases}
\end{align}
then the total number of gradient evaluations of $f_i$ performed by Algorithm~\ref{algVRASGD} to find a stochastic $\epsilon$-solution of \eqref{cp} can be bounded by
\beq\label{complexity-deter-unified}
\bar N := \begin{cases}
{\cal O}\left\{  m \log \tfrac{D_0}{\epsilon} \right\}, & m \ge \tfrac{D_0}{\epsilon} \mathrm{~or~} m\ge \tfrac{3L}{4\mu},\\
 {\cal O} \left\{m \log m +\sqrt{\tfrac{m D_0}{\epsilon}} \right\}, &   m < \tfrac{D_0}{\epsilon} \leq \tfrac{3L}{4\mu},\\
  {\cal O} \left\{m \log m +\sqrt{\tfrac{m L}{\mu}}\log \tfrac{D_0/\epsilon}{3L/4\mu} \right\}, & m < \tfrac{3L}{4\mu}  \le \tfrac{D_0}{\epsilon}.
\end{cases}
\eeq
where $D_0$ is defined as in \eqref{def_D_0}.
\end{theorem}

Observe that the complexity bound \eqref{complexity-deter-unified} is a unified convergence result for \algone to solve deterministic smooth convex finite-sum optimization problems \eqref{cp}.
When the strong convex modulus $\mu$ of the objective function is large enough, i.e., $3L/\mu<D_0/\epsilon$, \algone exhibits an 
optimal linear rate of convergence since the third case of \eqref{complexity-deter-unified} matches the lower bound \eqref{RIG_lb} for RIG methods.
If $\mu$ is relatively small, \algone treats the finite-sum problem \eqref{cp} as a smooth problem without strong convexity, which leads to the same 
complexity bounds as in Theorem~\ref{Them:main-deter-smooth}.
It should be pointed out that the parameter setting proposed in Theorem~\ref{Them:main-deter-sc} does not require the values of $\epsilon$ and $D_0$ given
a priori.

\subsection{Generalization of \algone}\label{sec:gen}
In this subsection, we extend \algone to solve two general classes of finite-sum optimization problems as well as establishing its convergence properties for these problems. 

\noindent{\bf Finite-sum problems under error bound condition.}
We investigate a class of weakly strongly convex problems, 
i.e., $\psi(x)$ is smooth convex and satisfies the error bound condition given by
\begin{equation}\label{cond:eb}
V(x,X^*)\leq \tfrac{1}{\bar{\mu}} (\psi(x)-\psi^*), ~\forall x\in X,
\end{equation}
where $X^*$ denotes the set of optimal solutions of \eqref{cp}.
Many optimization problems satisfy \eqref{cond:eb}, for instance, linear systems, quadratic programs, linear matrix inequalities and composite problems (outer: strongly convex, inner: polyhedron functions), see \cite{CongDDang:452} and Section 6 of \cite{necoara2018linear} for more examples.
Although these problems are not strongly convex, by properly restarting \algone we can solve them with an accelerated linear rate of convergence, the best-known complexity result to solve this class of problems so far. We formally present the result in Theorem~\ref{Them:main-deter-error}, whose proof is given in Appendix~\ref{sec:error_bound_proof}. 

\begin{theorem}[\bf Convex finite-sum optimization under error bound]\label{Them:main-deter-error}
Assume that the probabilities $q_i$'s are set to $L_i / \tsum_{i=1}^m L_i$ for $i = 1, \ldots, m$, and $\theta_t$ are defined as \eqref{def_theta_acc_SVRG}.
Moreover,
let us set parameters $\{\gamma_s\}$, $\{p_s\}$ and $\{\alpha_s\}$ as in \eqref{parameter-deter-smooth1} and \eqref{parameter-deter-alpha-sm} with $\{T_s\}$ being set as
\begin{align}\label{parameter-deter-error}
T_s = \begin{cases}
T_12^{s-1}, & s \le 4\\
8T_{1}, & s > 4
\end{cases},
\end{align}
where 
$T_1=\min\{m,\tfrac{L}{\bar{\mu}}\}$.
Then under condition \eqref{cond:eb}, for any $x^* \in X^*$, $s=4+4\sqrt{\tfrac{L}{\bar{\mu}m}}$,
\beq \label{results:acc_SVRG_error}
\bbe[\psi(\tilde{x}^s) - \psi(x^*)] \le
\tfrac{5}{16} [\psi({x}^0) - \psi(x^*)].
\eeq
Moreover, if we restart \algone every time it runs $s$ iterations for $k=\log\tfrac{\psi({x}^0) - \psi(x^*)}{\epsilon}$ times,
the total number of gradient evaluations of $f_i$ to find a stochastic $\epsilon$-solution of \eqref{cp} can be bounded by
\beq \label{results:acc_SVRG_error_sfo}
\bar N : =k(\tsum_s(m+T_s))=
{\cal O}\left\{\big(m+\sqrt{\tfrac{m L}{\bar{\mu}}}\big)\log \tfrac{\psi({x}^0) - \psi(x^*)}{\epsilon}\right\}.
\eeq
\end{theorem}
\begin{remark}
Note that \algone can also be extended to obtain an unified result as shown in Theorem~\ref{Them:main-deter-sc} for solving finite-sum problems under error bound condition. 
In particular, if the condition number is very large, i.e., $s= {\cal O} \{L/(\bar \mu m)\} \approx \infty$, \algone will never be restarted, and the resulting complexity bounds will reduce to 
the case for solving smooth convex problems provided in Theorem~\ref{Them:main-deter-smooth}.
\end{remark}

\noindent{\bf Stochastic finite-sum optimization.} 
We now consider stochastic smooth convex finite-sum optimization and online learning problems \todo{defined as in \eqref{sp}},
where only noisy gradient information of $f_i$ can be accessed via a SFO oracle.
In particular, for any $x\in X$, the SFO oracle outputs a vector $G_i(x,\xi_j)$ such that
\begin{align}
   &\E_{\xi_j}[G_i(x,\xi_j)]=\nabla f_i(x),~ i=1,\ldots,m, \label{eq:unbiased}\\
   &\E_{\xi_j}[\dnorms{G_i(x,\xi_j)-\nabla f_i(x)}] \leq \sigma^2,~ i=1,\ldots,m. \label{eq:bounded}
\end{align}
We present the variant of \algone for stochastic finite-sum optimization in Algorithm~\ref{algVRASGD_sto} as well as its convergence results in Theorem~\ref{Them:main-sto-smooth}, whose proof can be found in Appendix~\ref{sec:alg-stochastic}.
\begin{algorithm}[h]
  \caption{Stochastic variance-reduced accelerated gradient (Stochastic \newalgone)}
  \label{algVRASGD_sto}
  \begin{algorithmic}
\STATE
  \noindent This algorithm is the same as Algorithm~\ref{algVRASGD} except that for given batch-size parameters $B_s$ and $b_s$, Line~\ref{eqn:fullgradient} is replaced by $\tilde{x}=\tilde{x}^{s-1}$ and
  \beq\label{eqn:fullgradient_sto}
  \tilde{g} =\tfrac{1}{m}\tsum_{i=1}^{m}\left\{G_i(\tilde{x}):=\tfrac{1}{B_s}\tsum_{j=1}^{B_s}G_i(\tilde{x},\xi_{j}^s)\right\},
  \eeq
    \STATE and Line~\ref{eqn:estgradient} is replaced by
  \beq\label{eqn:estgradient_sto}
  G_t = \tfrac{1}{q_{i_t} m b_s}\tsum_{k=1}^{b_s}\big(G_{i_t}(\underline x_{t},\xi_{k}^s)
            - G_{i_t}(\tilde{x})\big)
            + \tilde{g}.
  \eeq
  \end{algorithmic}
\end{algorithm}
\begin{theorem}[\bf Stochastic smooth finite-sum optimization]\label{Them:main-sto-smooth}
Assume that $\theta_t$ are defined as in \eqref{def_theta_acc_SVRG},
$C:=\tsum_{i=1}^m{\tfrac{1}{q_im^2}}$ and the probabilities $q_i$'s are set to $L_i / \tsum_{i=1}^m L_i$ for $i = 1, \ldots, m$.
Moreover,
let us denote $s_0 := \lfloor \log m\rfloor+1$ and set $T_s$, $\alpha_s$, $\gamma_s$ and $p_s$ as in \eqref{parameter-deter-smooth1} and \eqref{parameter-deter-alpha-sm}.
Then the number of calls to the SFO oracle required by 
Algorithm \ref{algVRASGD_sto} to find a stochastic $\epsilon$-solution of \eqref{cp} can be bounded by
\beq\label{complexity-sto-sfo}
N_\mathrm{SFO} = \tsum_s (mB_s+T_sb_s) =
\begin{cases}
{\cal O}\left\{  \tfrac{mC\sigma^2}{L\epsilon} \right\}, & m \ge D_0/\epsilon,\\
 {\cal O} \left\{\tfrac{C\sigma^2D_0}{L\epsilon^2} \right\}, &  m < D_0/\epsilon,
\end{cases}
\eeq
where $D_0$ is given in \eqref{def_D_0}.
\end{theorem}
\begin{remark}
Note that the constant $C$ in \eqref{complexity-sto-sfo} can be easily upper bounded by $\tfrac{L}{\min\{L_i\}}$, and $C=1$ if $L_i = L, \forall i$. 
To the best of our knowledge, among a few existing RIG methods that can be applied to solve the class of stochastic finite-sum problems, \algone is the first to achieve such complexity results as in \eqref{complexity-sto-sfo} for smooth convex problems. RGEM\cite{lan2018random} obtains nearly-optimal rate of convergence for strongly convex case, but cannot solve stochastic smooth problems directly, and \cite{kulunchakov2019estimate} required a specific initial point, i.e., an exact solution to a proximal mapping depending on the variance $\sigma^2$, to achieve ${\cal O}\left\{m\log m + \sigma^2/\epsilon^2\right\}$ rate of convergence for smooth convex problems.
\end{remark}
\setcounter{equation}{0}

\section{Numerical experiments}\label{sec:numerical}
In this section, we demonstrate the advantages of our proposed algorithm, \algone over several state-of-the-art algorithms, e.g., SVRG++ \cite{allen2016improved} and Katyusha \cite{allen2016katyusha}, etc., via solving several well-known machine learning models. For all experiments, we use public real datasets downloaded from UCI Machine Learning Repository \cite{Dua:2019} \todo{and uniform sampling strategy to select $f_i$. Indeed, the theoretical suggesting sampling distribution should be non-uniform, i.e., $q_i  = L_i/\tsum_{i=1}^m L_i$, which results in the optimal constant $L$ appearing in the convergence results. However, a uniform sampling strategy will only lead to a constant factor slightly larger than $L=\tfrac{1}{m}\tsum_{i=1}^m L_i$. Moreover, it is computationally efficient to estimate $L_i$ by performing maximum singular value decomposition of the Hessian since only a rough estimation suffices.} 

\noindent{\bf Unconstrained smooth convex problems.} 
We first investigate unconstrained logistic models which cannot be solved via the perturbation approach due to the unboundedness of the feasible set. More specifically, we applied \newalgone, SVRG++ and Katyusha\textsuperscript{ns} to solve a logistic regression problem,
\beq\label{p-logistic}
\min_{x\in \bbr^n} \{\psi(x):= \tfrac{1}{m}\tsum_{i=1}^m f_i(x)\} \ \mbox{where }  f_i(x): = \log(1+\exp(-b_ia_i^Tx))\}.
\eeq
Here $(a_i,b_i)\in \bbr^{n}\times\{-1,1\}$ is a training data point and $m$ is the sample size, and hence $f_i$ now corresponds to the loss generated by a single training data. 
As we can see from Figure~\ref{logistic-results}, \algone converges much faster than SVRG++ and Katyusha in terms of training loss. 
\begin{figure}[H]
\begin{minipage}{0.45\textwidth}
\centering
\includegraphics[scale = 0.15]{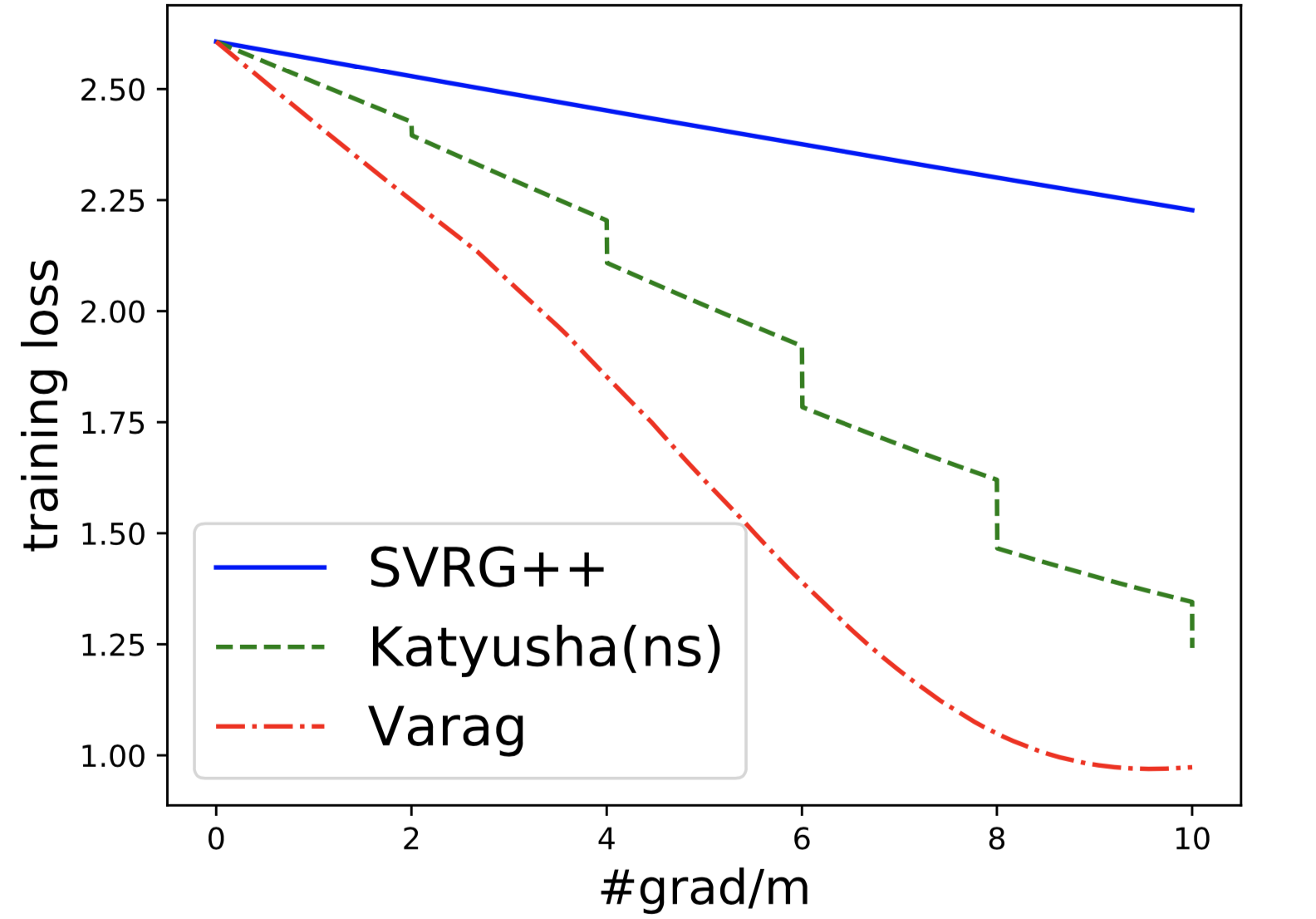}
\\
{\small Diabetes ($m=1151$), \\ unconstrained logistic}
\end{minipage} 
\begin{minipage}{0.45\textwidth}
\centering
\includegraphics[scale = 0.15]{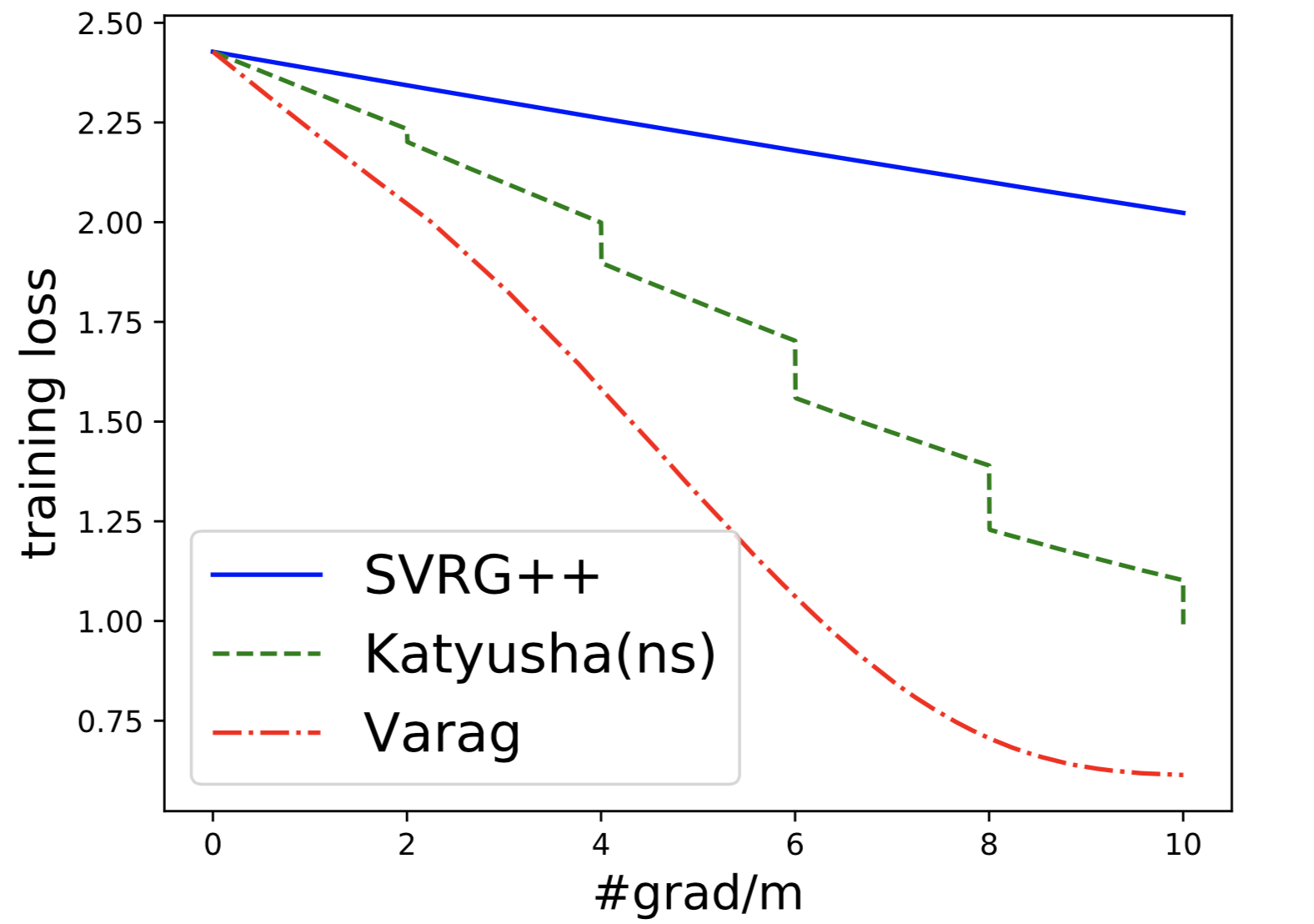}
\\
{\small  Breast Cancer Wisconsin ($m=683$), unconstrained logistic}
\end{minipage}
\caption{\footnotesize The algorithmic parameters for SVRG++ and Katyusha\textsuperscript{ns} are set according to \cite{allen2016improved} and \cite{allen2016katyusha}, respectively, and those for \algone are set as in Theorem~\ref{Them:main-deter-smooth}. }
\label{logistic-results}
\end{figure}
\vspace*{-.2in}
\noindent{\bf Strongly convex loss with simple convex regularizer.}
We now study the class of Lasso regression problems with $\lambda$ as the regularizer coefficient, given in the following form
\beq\label{p-lasso}
\min_{x\in \bbr^n}\{\psi(x):= \tfrac{1}{m}\tsum_{i=1}^{m}f_i(x) + h(x)\} \ \mbox{where } f_i(x):=\tfrac{1}{2}(a_i^Tx-b_i)^2, h(x):=\lambda \|x\|_1.
\eeq
Due to the assumption SVRG++ and Katyusha enforced on the objective function that the strong convexity can only be associated with the regularizer, these methods always view Lasso as smooth problems \cite{tang2018rest},  while \algone can treat Lasso as strongly convex problems.
As can be seen from Figure~\ref{lasso-results}, \algone outperforms SVRG++ and Katyusha\textsuperscript{ns} in terms of training loss. 
\begin{figure}[H]
\begin{minipage}{0.45\textwidth}
\centering
\includegraphics[scale = 0.15]{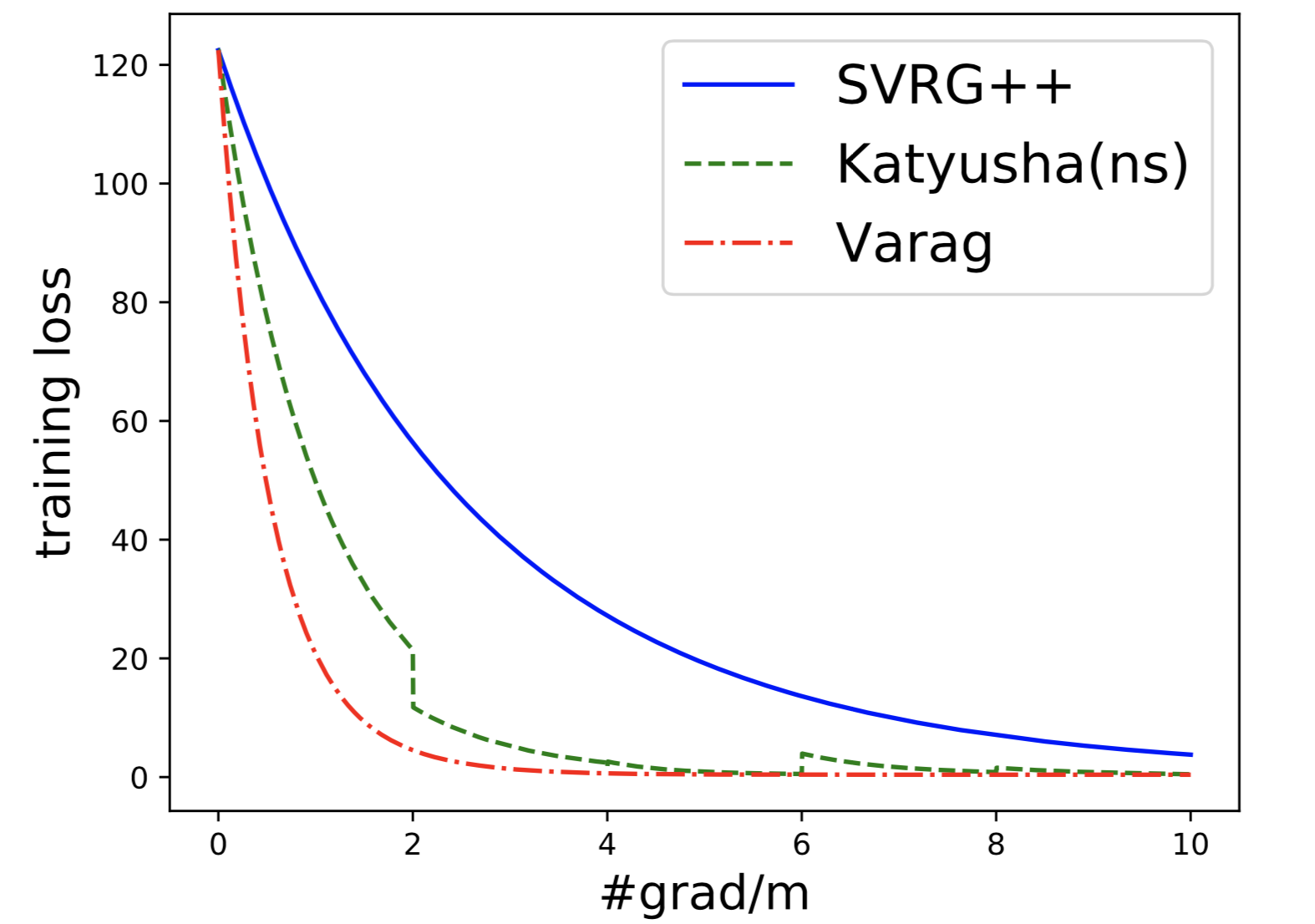}
\\
{\small  Diabetes ($m=1151$), \\ Lasso $\lambda=0.001$}
\end{minipage} 
\begin{minipage}{0.45\textwidth}
\centering
\includegraphics[scale = 0.15]{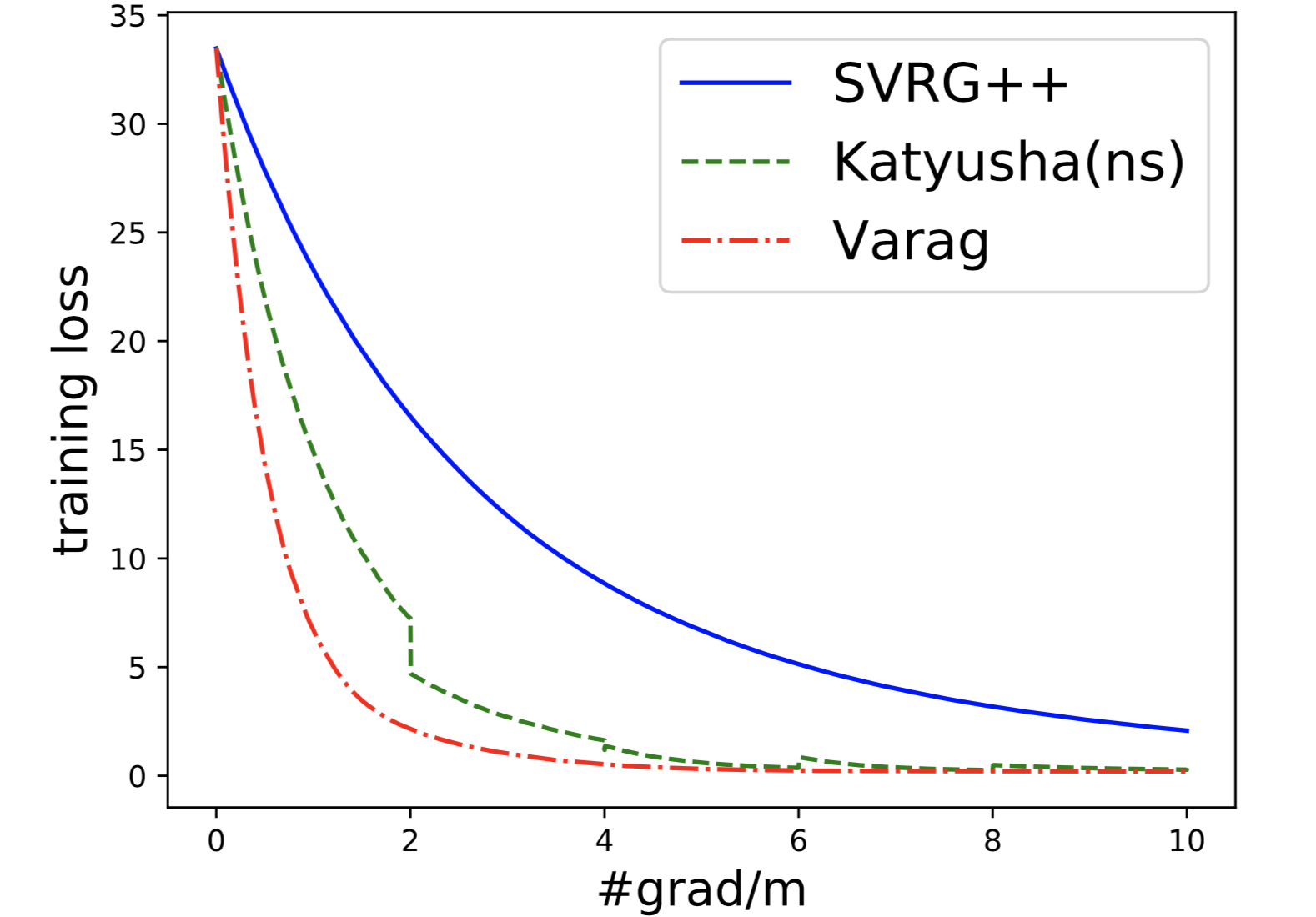}
\\
{\small  Breast Cancer Wisconsin ($m=683$), \\ Lasso $\lambda=0.001$}
\end{minipage}
\caption{\footnotesize The algorithmic parameters for SVRG++ and Katyusha\textsuperscript{ns} are set according to \cite{allen2016improved} and \cite{allen2016katyusha}, respectively, and those for \algone are set as in Theorem~\ref{Them:main-deter-sc}. }
\label{lasso-results}
\end{figure}
\vspace*{-.2in}
\noindent{\bf Weakly strongly convex problems satisfying error bound condition.}
Let us consider a special class of finite-sum convex quadratic problems given in the following form
\begin{align}\label{p-QP}
\min_{x\in \bbr^n}\{\psi(x):=\tfrac{1}{m}\tsum_{i=1}^m f_i(x)\} \ \mbox{where }  f_i(x): = \tfrac{1}{2}x^TQ_ix+q_i^Tx.
\end{align}
Here $q_i=-Q_ix_s$ and $x_s$ is a solution to the symmetric linear system $Q_ix+q_i=0$ with $Q_i\succeq 0$.
\cite{CongDDang:452}[Section 6] and \cite{necoara2018linear}[Section 6.1] proved that \eqref{p-QP} belongs to the class of weakly strongly convex problems satisfying error bound condition \eqref{cond:eb}. For a given solution $x_s$, we use the following real datasets to generate $Q_i$ and $q_i$. We then compare the performance of \algone with fast gradient method (FGM) proposed in \cite{necoara2018linear}.  
As shown in Figure~\ref{errorbound-results}, \algone outperforms FGM for all cases. And as the number of component functions $m$ increases, \algone demonstrates more advantages over FGM. 
These numerical results are consistent with the theoretical complexity bound \eqref{results:acc_SVRG_error_sfo} suggesting that \algone can save up to ${\cal O}\{\sqrt m\}$ number of gradient computations than deterministic algorithms, e.g., FGM.
\begin{figure}[H]
\begin{minipage}{0.45\textwidth}
\centering
\includegraphics[scale = 0.15]{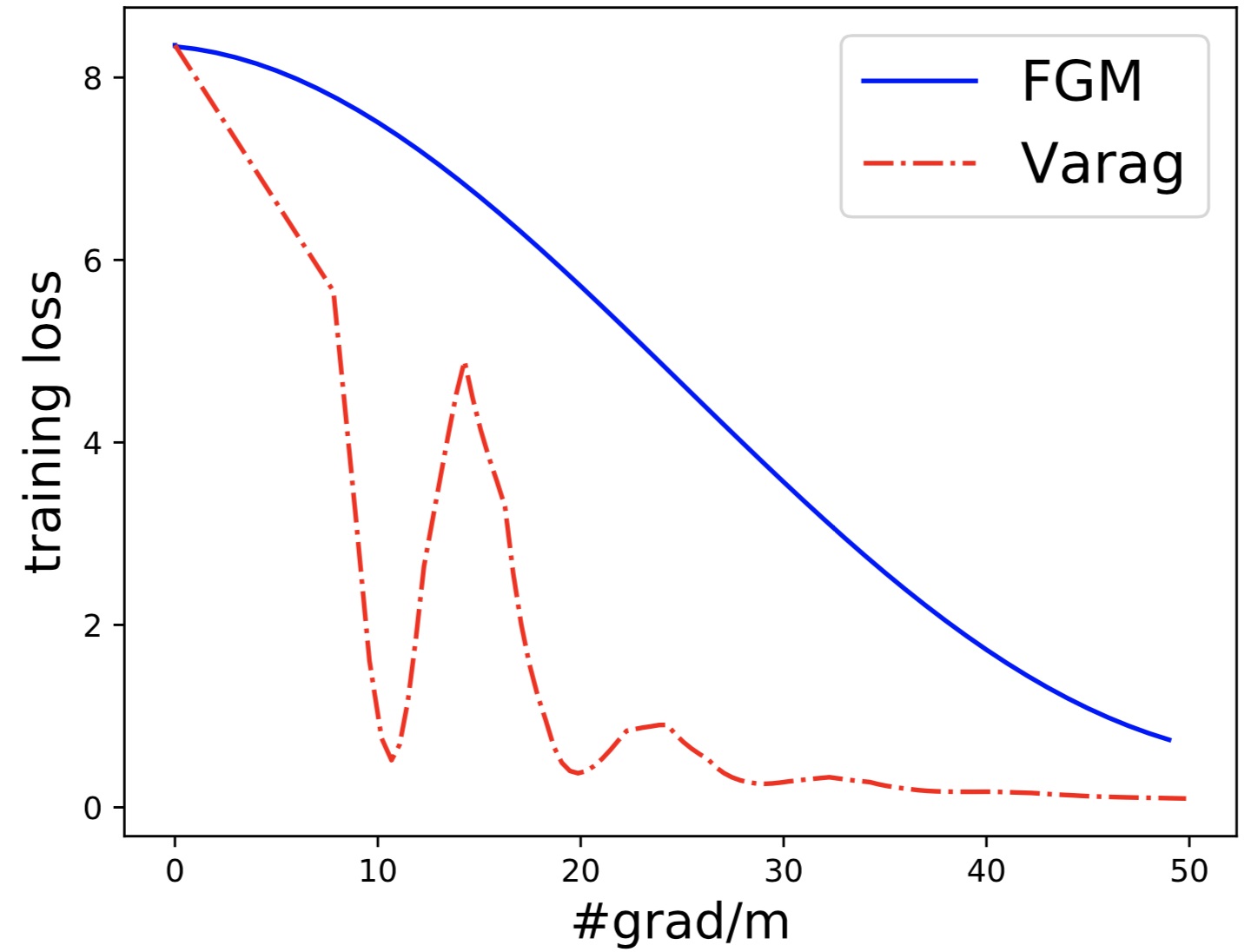}
\\
{\small Diabetes ($m=1151$)}
\end{minipage}
\begin{minipage}{0.45\textwidth}
\centering
\includegraphics[scale = 0.15]{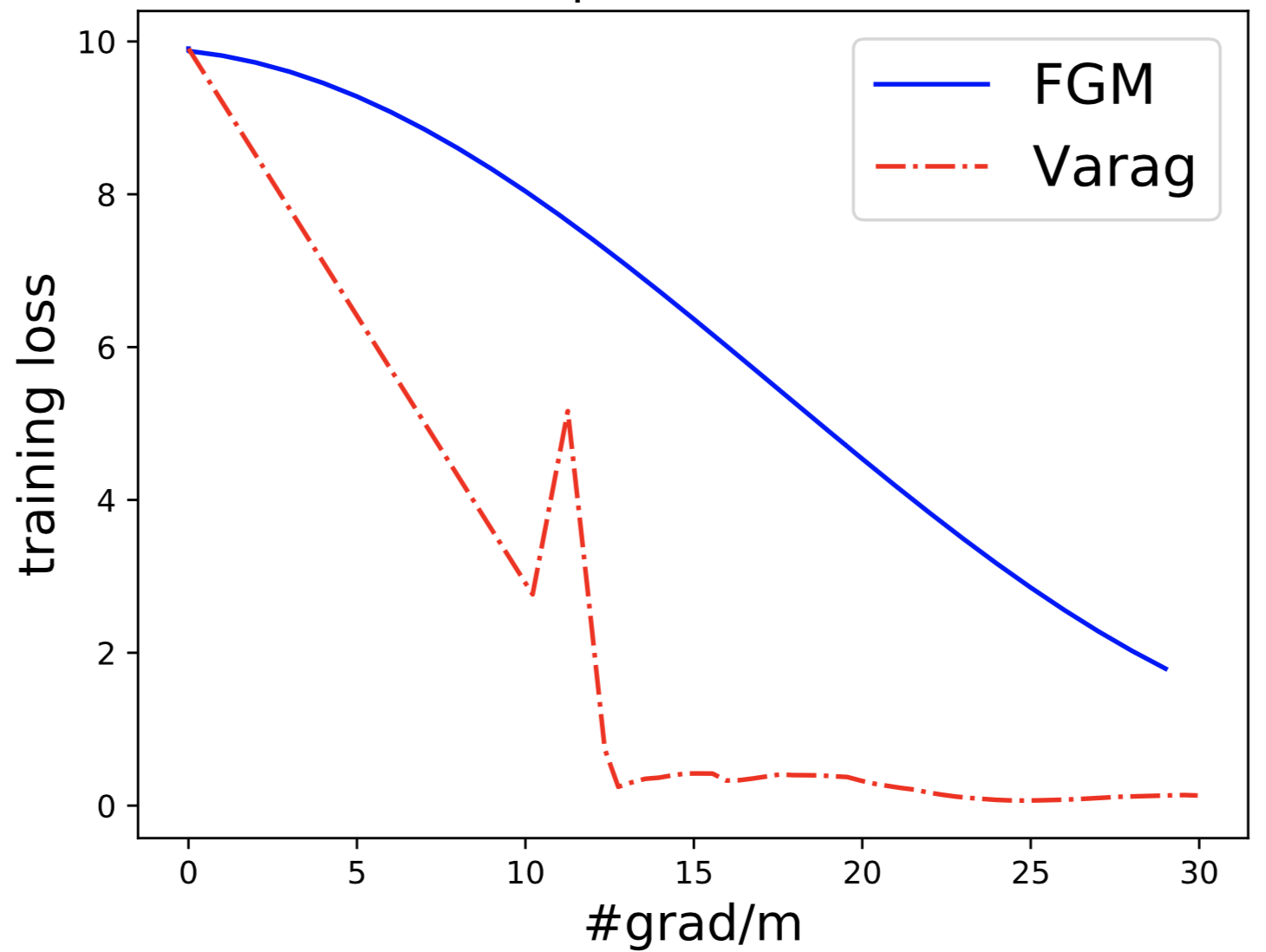}
\\
{\small Parkinsons Telemonitoring ($m=5875$)}
\end{minipage}
\caption{\footnotesize The algorithmic parameters for FGM and \algone are set according to \cite{necoara2018linear} and Theorem~\ref{Them:main-deter-error}, respectively. }
\label{errorbound-results}
\end{figure}
\vspace*{-.2in}
More numerical experiment results on another problem case, strongly convex problems with small strongly convex modulus, can be found in Appendix~\ref{sec:experiments}.
\setcounter{equation}{0}

\newpage
\bibliographystyle{plain}
\bibliography{ref-fastVR}

\newpage
\appendix
\section{Convergence analysis of \algone for deterministic finite-sum optimization}\label{sec:analysis_deter}
Our main goal in this section is to establish the convergence results stated in Theorems \ref{Them:main-deter-smooth} and~\ref{Them:main-deter-sc}
for the \algone method applied to the finite-sum optimization problem in \eqref{cp}.

Before proving Theorem \ref{Them:main-deter-smooth} and~\ref{Them:main-deter-sc}, we first need to present some basic properties for smooth convex functions and then provide some important technical results.

\vgap
\begin{lemma} \label{smoothfunction_prop}
	If $f: X \to \bbr$ has Lipschitz continuous gradients with Lipschitz constant $L$, then
	\[
	\tfrac{1}{2L} \|\nabla f(x) - \nabla f(z)\|_*^2 \leq
	f(x) - f(z) - \langle \nabla f(z), x-z \rangle \ \ \forall x, z \in X.
	\]
\end{lemma}

\begin{proof}
	Denote
	$\phi(x) = f(x) - f(z) - \langle \nabla f(z), x-z \rangle$. Clearly $\phi$ also has $L$-Lipschitz continuous gradients.
	It is easy to check that $\nabla \phi(z)=0$, and hence that $\min_x \phi(x) =\phi(z)= 0$, which implies
	\begin{align}
	\phi(z)& \leq \phi(x-\tfrac{1}{L}\nabla\phi(x)) \notag\\
	&=\phi(x)+\int_0^1\langle \nabla\phi\left(x-\tfrac{\tau}{L}\nabla \phi(x)\right),-\tfrac{1}{L}\nabla \phi(x) \rangle d\tau \notag\\
	&=\phi(x)+\langle \nabla \phi(x),-\tfrac{1}{L}\nabla \phi(x)\rangle+\int_0^1\langle \nabla \phi\left(x-\tfrac{\tau}{L}\nabla \phi(x)\right)-\nabla \phi(x),-\tfrac{1}{L} \nabla \phi(x) \rangle d\tau \notag\\
	& \leq \phi(x)-\tfrac{1}{L}\|\nabla \phi(x)\|_*^2+\int_0^1 L\dnorm{\tfrac{\tau}{L}\nabla \phi(x)}~\dnorm{\tfrac{1}{L} \nabla \phi(x)} d\tau \notag\\
	&= \phi(x)-\tfrac{1}{2L}\dnorm{\nabla \phi(x)}^2. \notag
	\end{align}
	Therefore, we have
	$
	\tfrac{1}{2L} \dnorm{\nabla \phi(x)}^2 \leq \phi(x)  - \phi(z) = \phi(x),
	$
	and the result follows immediately from this relation.
\end{proof}

The following result follows as a consequence of Lemma~\ref{smoothfunction_prop}.

\begin{lemma}\label{bound}
	Let $x^*$ be an optimal solution of~\eqref{cp}. Then we have
	\begin{align}
	\tfrac{1}{m}\tsum_{i=1}^m \tfrac{1}{m q_i} \dnorm{\nabla f_i(x)-\nabla f_i(x^*)}^2
	\leq 2 L_Q\left[ \psi(x) - \psi(x^*)\right], \ \forall x \in X,
	\end{align}
	where
	\beq \label{def_LQ}
	L_Q= \tfrac{1}{m} \max_{i=1, \ldots,m} \tfrac{L_i}{q_i}.
	\eeq
\end{lemma}

\begin{proof}
	By Lemma~\ref{smoothfunction_prop} (with $f=f_i$), we have
	$$
	\dnorm{\nabla f_i(x) - \nabla f_i(x^*)}^2 \leq
	2L_i \left[f_i(x) - f_i(x^*) -\langle \nabla f_i(x^*), x-x^*\rangle\right].
	$$
	Dividing this inequality by $1/(m^2 q_i)$,
	and summing over $i=1,\ldots,m$, we obtain
	\begin{align}
	\tfrac{1}{m}\tsum_{i=1}^m\tfrac{1}{m q_i}\dnorm{\nabla f_i(x) - \nabla f_i(x^*)}^2
	\leq 2L_Q \left[ f(x) - f(x^*) - \langle \nabla f(x^*), x-x^*\rangle\right]. \label{SVRG_bnd0}
	\end{align}
	By the optimality of $x^*$, we have $\langle \nabla f(x^*) + h'(x^*), x - x^*\rangle \ge 0$ for any $x \in X$,
	which in view of the convexity of $h$, implies that
	$\langle \nabla f(x^*), x - x^*\rangle \ge h(x^*) - h(x)$ for any $x \in X$.
	The result then follows by combining the previous two conclusions.
\end{proof}

\vgap

In the sequel, let us define some important notations that help us to simplify the convergence analysis of \algone.
\begin{align}
l_f(z, x) &:= f(z) + \langle \nabla f(z), x - z\rangle, \label{def_linear_f}\\
\delta_t &:= G_t - \nabla f (\underline x_t), \label{def_VARSG_delta}\\
x^+_{t-1} &:= \tfrac{1}{1+\mu \gamma_s}\left( x_{t-1} + \mu \gamma_s \underline x_t \right), \label{def_x_t_plus}
\end{align}
where $G_t$, $\underline x_t$ and $x_{t-1}$ are generated as in Algorithm~\ref{algVRASGD}.
Lemma~\ref{variance_reduced_betterEsti} below shows that $G_t$ is an unbiased estimator of $\nabla f(\underline x_t)$ and provides a tight upper bound for its variance.
\begin{lemma}\label{variance_reduced_betterEsti}
	Conditionally on~$x_1, \ldots, x_{t-1}$,
	\begin{align}
	\bbe [\delta_t] &=0, \label{SVRG_bnd1}\\
	\bbe[\dnorm{\delta_t}^2] &\le 2 L_Q[f(\tilde x) - f(\underline x_{t}) - \langle \nabla f(\underline x_{t}), \tilde x - \underline x_{t}\rangle]. \label{SVRG_bnd2}
	\end{align}
\end{lemma}

\begin{proof}  We take the expectation with respect to~$i_t$ conditionally on~$x_1, \ldots, x_{t}$, to obtain
	\begin{align*}
	\bbe \left[\tfrac{1}{m q_{i_t}}\nabla f_{i_t}(\underline x_{t})\right]
	= \tsum_{i=1}^m \tfrac{q_i}{m q_i} \nabla f_i(\underline x_{t})
	= \tsum_{i=1}^m \tfrac{1}{m} \nabla f_i(\underline x_{t})
	=\nabla f(\underline x_{t}).
	\end{align*}
	Similarly we have
	$\bbe \left[\tfrac{1}{m q_{i_t}}\nabla f_{i_t}(\tilde{x})\right]=\nabla f(\tilde{x})$.
	Therefore,
	\begin{align*}
	\bbe [G_t] = \bbe \left[\tfrac{1}{m q_{i_t}}\big(\nabla f_{i_t}(\underline x_{t}) - \nabla f_{i_t}(\tilde{x})\big) + \nabla f(\tilde{x}) \right]
	= \nabla f(\underline x_{t}).
	\end{align*}
	To bound the variance, we have
	\begin{align*}
	\bbe [\dnorm{\delta_t}^2] =& \bbe [\dnorm{\tfrac{1}{m q_{i_t}}\big(\nabla f_{i_t} (\underline x_{t})-\nabla f_{i_t}(\tilde{x})\big) +\nabla f(\tilde{x}) - \nabla f(\underline x_{t}) }^2]  \notag\\
	=& \bbe [\tfrac{1}{(m q_{i_t})^2}\dnorm{\nabla f_{i_t} (\underline x_{t}) - \nabla f_{i_t}(\tilde{x}) }^2]
	- \dnorm{\nabla f(\underline x_{t}) - \nabla f(\tilde{x})}^2 \notag\\
	\leq& \bbe [\tfrac{1}{(m q_{i_t})^2}\dnorm{\nabla f_{i_t} (\underline x_{t}) - \nabla f_{i_t}(\tilde{x}) }^2]
	\end{align*}
	The above relation, in view of relation~\eqref{SVRG_bnd0} (with $x$ and $x^*$ replaced by $\tilde x$ and $\underline x_{t}$), then implies \eqnok{SVRG_bnd2}.
\end{proof}

Using the definition of $x^+_{t-1}$ in \eqref{def_x_t_plus}, and the definitions of $\underline x_t$ and $\bar x_t$ in Algorithm~\ref{algVRASGD} (see Line~\ref{eqn:underlinex} and \ref{eqn:barx}), we have
\begin{align}\label{eqn:barx-xplus}
\bar x_t - \underline x_t & = (1 - \alpha_s-p_s) \bar x_{t-1}  + \alpha_s x_{t} + p_s \tilde x - \underline x_t \nn\\
&=   \alpha_s x_{t} + \tfrac{1}{1+\mu \gamma_s} \left\{[1+\mu \gamma_s(1-\alpha_s)] \underline x_t
- \alpha_s x_{t-1}\right\} - \underline x_t\nn\\
&= \alpha_s (x_t - x^+_{t-1}).
\end{align}

We characterize the solutions of the prox-mapping \eqref{prox_mapping} (or Line~\ref{eqn:xt} of Algorithm~\ref{algVRASGD}) in Lemma~\ref{tech1_prox} below.

\begin{lemma}[{\cite[Lemma 2]{GhaLan12-2a}}] \label{tech1_prox}
	Let the convex function $p: X\to \bbr$, the points $\tilde x, \tilde y \in X$ and the scalars
	$\mu_1,\mu_2\geq 0$ be given.
	Let $w:X \to \bbr$ be a convex function and $V(x^0,x)$ be defined in \eqref{primal_prox}.
	If
	\beq \label{tech_lemma_prob}
	u^* \in \Argmin \{ p(u) + \mu_1 V(\tilde x, u) + \mu_2 V(\tilde y, u) : u \in X\},
	\eeq
	then for any $u \in X$, we have
	\[
	p(u^*) + \mu_1 V(\tilde x, u^*) + \mu_2 V(\tilde y, u^*)
	\le p(u) + \mu_1 V(\tilde x, u) + \mu_2 V(\tilde y, u)- (\mu_1 + \mu_2) V(u^*,u).
	\]
\end{lemma}

The following result examines the optimality conditions associated with the definition of $x_t$ in Line~\ref{eqn:xt} of Algorithm~\ref{algVRASGD}.

\begin{lemma} \label{lem_VRASGD_opt_cond}
	For any $x \in X$, we have
	\begin{align*}
	\gamma_s [l_f(\underline x_t, x_t) - l_f(\underline x_t, x) + h(x_t) - h(x)] \le \gamma_s \mu V(\underline x_t, x) + V(x_{t-1}, x) - (1+\mu \gamma_s) V(x_t, x) \\
	- \tfrac{1+\mu \gamma_s}{2} \|x_t - x^+_{t-1}\|^2 - \gamma_s \langle \delta_t, x_t - x \rangle.
	\end{align*}
\end{lemma}
\begin{proof}
	It follows from Lemma \ref{tech1_prox} and the definition of $x_t$ in Algorithm~\ref{algVRASGD} that
	\begin{align*}
	\gamma_s[\langle G_t, x_t - x \rangle + h(x_t) - h(x) +\mu V(\underline x_t, x_t)] +  + V(x_{t-1}, x_t) \nn\\
	\le \gamma_s \mu V(\underline x_t, x) + V(x_{t-1}, x) - (1+\mu \gamma_s) V(x_t, x).
	\end{align*}
	Also observe that
	\begin{align*}
	\langle G_t, x_t - x \rangle &= \langle \nabla f(\underline x_t), x_t - x \rangle + \langle \delta_t, x_t - x\rangle
	= l_f(\underline x_t, x_t) - l_f(\underline x_t, x) + \langle \delta_t, x_t - x \rangle
	\end{align*}
	and
	\begin{align*}
	\gamma_s \mu V(\underline x_t, x_t) + V(x_{t-1}, x_t) &\ge \tfrac{1}{2} \left( \mu \gamma_s \|x_t - \underline x_t\|^2 + \|x_t - x_{t-1}\|^2\right)\\
	&\ge \tfrac{1+\mu \gamma_s}{2} \|x_t - x^+_{t-1}\|^2,
	\end{align*}
	where the last inequality follows from the definition of $x^+_{t-1}$ in \eqnok{def_x_t_plus} and the convexity of $\|\cdot\|$.
	The result then follows by combining the above three relations.
\end{proof}

\vgap
We now show the possible progress made by each inner iteration of the \algone method.
\begin{lemma} \label{Lemma_VARSGD_iter}
	Assume that $\alpha_s \in [0,1]$, $p_s \in [0,1]$ and $\gamma_s > 0$ satisfy
	\begin{align}
	1+\mu \gamma_s - L \alpha_s \gamma_s > 0, \label{VRASGD_cond1}\\
	p_s -  \tfrac{L_Q \alpha_s \gamma_s}{1+\mu \gamma_s - L \alpha_s \gamma_s} \ge  0. \label{VRASGD_cond2}
	\end{align}
	Then, conditional on $x_1, \ldots, x_{t-1}$, we have
	\begin{align} \label{eq:lemma_iter}
	\tfrac{\gamma_s}{\alpha_s}\bbe[\psi(\bar x_t) - \psi(x)]&+ (1+\mu \gamma_s)\bbe[V(x_t,x)]\nn \\
	&\le \tfrac{\gamma_s}{\alpha_s}(1-\alpha_s - p_s) [\psi(\bar x_{t-1})-\psi(x)] + \tfrac{\gamma_sp_s}{\alpha_s}[\psi(\tilde x)-\psi(x)]
	+ V(x_{t-1},x)
	\end{align}
	for any $x \in X$.
\end{lemma}
\begin{proof}
	Note that by the smoothness of $f$, the definition of $\bar x_t$, and \eqref{eqn:barx-xplus}, we have
	\begin{align*}
	f(\bar x_t) &\le l_f(\underline x_t, \bar x_t) + \tfrac{L}{2} \|\bar x_t - \underline x_t\|^2 \\
	&= (1-\alpha_s - p_s) l_f(\underline x_t, \bar x_{t-1}) + \alpha_s l_f(\underline x_t, x_t) + p_s l_f(\underline x_t, \tilde x)
	+\tfrac{L\alpha_s^2}{2} \|x_t - x_{t-1}^+\|^2.
	\end{align*}
	The above inequality, in view of Lemma~\ref{lem_VRASGD_opt_cond} and the (strong) convexity of $f$, then implies that
	\begin{align}
	f(\bar x_t) &\le (1-\alpha_s - p_s) l_f(\underline x_t, \bar x_{t-1}) \nn\\
	&\quad+
	\alpha_s \left[l_f(\underline x_t, x) + h(x) - h(x_t) + \mu V(\underline x_t, x) + \tfrac{1}{\gamma_s} V(x_{t-1},x) - \tfrac{1+\mu \gamma_s}{\gamma_s} V(x_t,x)
	\right]\nn\\
	&\quad + p_s l_f(\underline x_t, \tilde x) - \tfrac{\alpha_s}{2\gamma_s} (1+\mu \gamma_s - L \alpha_s \gamma_s)\|x_t - x_{t-1}^+\|^2 - \alpha_s \langle \delta_t, x_t - x\rangle \nn \\
	&\le (1-\alpha_s - p_s) f(\bar x_{t-1})
	+ \alpha_s \left[\psi(x) - h(x_t) + \tfrac{1}{\gamma_s} V(x_{t-1},x) - \tfrac{1+\mu \gamma_s}{\gamma_s} V(x_t,x) \right] \nn \\
	&\quad + p_s l_f(\underline x_t, \tilde x) - \tfrac{\alpha_s}{2\gamma_s} (1+\mu \gamma_s - L \alpha_s \gamma_s)\|x_t - x_{t-1}^+\|^2\nn\\
	&\quad - \alpha_s \langle \delta_t, x_t - x_{t-1}^+\rangle - \alpha_s \langle \delta_t,x_{t-1}^+ - x\rangle\nn \\
	&\le  (1-\alpha_s - p_s) f(\bar x_{t-1})
	+ \alpha_s \left[\psi(x) - h(x_t) + \tfrac{1}{\gamma_s} V(x_{t-1},x) - \tfrac{1+\mu \gamma_s}{\gamma_s} V(x_t,x) \right]\nn \\
	&\quad +  p_s l_f(\underline x_t, \tilde x) + \tfrac{\alpha_s \gamma_s\|\delta_t\|_*^2}{2(1+\mu \gamma_s - L \alpha_s \gamma_s)}
	- \alpha_s \langle \delta_t,x_{t-1}^+ - x\rangle,\label{rel1}
	\end{align}
	where the last inequality follows from the fact that $b\langle u,v\rangle - a\|v\|^2/2 \le b^2\|u\|^2/(2a), \forall a>0$.
	Note that by \eqnok{SVRG_bnd1}, \eqnok{SVRG_bnd2}, \eqnok{VRASGD_cond2} and the convexity of $f$, we have, conditional on $x_1, \ldots, x_{t-1}$,
	\begin{align*}
	& p_s l_f(\underline x_t, \tilde x) + \tfrac{\alpha_s \gamma_s\bbe[\|\delta_t\|_*^2]}{2(1+\mu \gamma_s - L \alpha_s \gamma_s)} - \alpha_s \bbe[\langle \delta_t,x_{t-1}^+ - x\rangle]\\
	& \le  p_s l_f(\underline x_t, \tilde x) + \tfrac{L_Q\alpha_s \gamma_s}{1+\mu \gamma_s - L \alpha_s \gamma_s} [f(\tilde x) - l_f(\underline x_t, \tilde x) ]\\
	&\le \left(p_s -  \tfrac{L_Q \alpha_s \gamma_s}{1+\mu \gamma_s - L \alpha_s \gamma_s}\right) l_f(\underline x_t, \tilde x)
	+  \tfrac{L_Q \alpha_s \gamma_s}{1+\mu \gamma_s - L \alpha_s \gamma_s}f(\tilde x)
	\le p_s f(\tilde x).
	\end{align*}
	Moreover, by convexity of $h$, we have $h(\bar x_t) \le (1-\alpha_s - p_s) h(\bar x_{t-1}) + \alpha_s h(x_t) + p_s h(\tilde x)$.
	Summing up the previous three conclusions, we obtain
	\begin{align*}
	\bbe[\psi(\bar x_t)+ \tfrac{\alpha_s(1+\mu \gamma_s)}{\gamma_s}V(x_t,x)] \le (1-\alpha_s - p_s) \psi(\bar x_{t-1}) + p_s \psi(\tilde x) + \alpha_s \psi(x)
	+ \tfrac{\alpha_s}{\gamma_s}V(x_{t-1},x).
	\end{align*}
	The result then follows by subtracting $\psi(x)$ from both sides of the above inequality.
\end{proof}

\subsection{Smooth convex problems}\label{sec:deter-sm-pf}
In this subsection, we assume that $f$ is not necessarily strongly convex, i.e., $\mu = 0$ in \eqref{def_strongconvexity}.
Lemma \ref{lem:deter_smooth_one_epoch} below
shows possible decrease of functional value
in each epoch of \algone for solving these problems.

\begin{lemma}\label{lem:deter_smooth_one_epoch}
	Assume that for each epoch $s$, $s \ge 1$, the parameters $\alpha_s$, $\gamma_s$, $p_s$
	and $T_s$ are chosen such that \eqnok{VRASGD_cond1}-\eqnok{VRASGD_cond2} hold. Also, let us set
	$\theta_t$ to \eqref{def_theta_acc_SVRG}.
	Moreover, let us denote
	\beq \label{def_LHs}
	{\cal L}_s :=  \tfrac{\gamma_s}{\alpha_s} + (T_s -1)  \tfrac{\gamma_s (\alpha_s + p_s)}{\alpha_s},\
	{\cal R}_s := \tfrac{\gamma_s}{\alpha_s} (1-\alpha_s) + (T_s-1) \tfrac{\gamma_s p_s}{\alpha_s},
	\eeq
	and assume that
	\beq\label{def_ws}
	w_s := {\cal L}_s - {\cal R}_{s+1} \ge 0, \forall s \ge 1.
	\eeq
	Then we have
	\begin{align}
	&{\cal L}_s \bbe[\psi(\tilde{x}^s) - \psi(x)] + (\tsum_{j=1}^{s-1} w_j)
	\bbe[\psi(\bar x^s) - \psi(x)]   \nn \\
	&\quad \quad \le {\cal R}_1 \bbe[\psi(\tilde{x}^{0}) - \psi(x)]
	+\bbe[V(x^{0},x) - V(x^s,x)]  \label{VRASGD_smooth}
	\end{align}
	for any $x \in X$, where
	\beq \label{def_w_output}
	\bar x^s :=  (\tsum_{j=1}^{s-1} w_j) \tsum_{j=1}^{s-1}(w_j \tilde x^j).
	\eeq
\end{lemma}
\begin{proof}
	Using our assumptions on $\alpha_s$, $\gamma_s$ and $p_s$, and
	the fact that $\mu = 0$,
	we have
	\begin{align*}
	\tfrac{\gamma_s}{\alpha_s}  \bbe[\psi(\bar x_t) - \psi(x)] &\le  \tfrac{\gamma_s}{\alpha_s} (1-\alpha_s - p_s) \bbe[\psi(\bar x_{t-1}) - \psi(x)] \\
	&\quad+  \tfrac{\gamma_s p_s}{\alpha_s}  \bbe[\psi(\tilde x) - \psi(x)]
	+\bbe[V(x_{t-1},x) - V(x_t,x)].
	\end{align*}
	Summing up these inequalities for $t=1, \ldots, T_s$, using the definition of $\theta_t$ in \eqnok{def_theta_acc_SVRG} and
	the fact that $\bar x_0 = \tilde x$, and rearranging the terms, we have
	\begin{align*}
	\tsum_{t=1}^{T_s} \theta_t  \bbe[\psi(\bar x_t) - \psi(x)]
	& \le  \left[\tfrac{\gamma_s}{\alpha_s} (1-\alpha_s) + (T_s-1) \tfrac{\gamma_s p_s}{\alpha_s}\right] \bbe[\psi(\tilde x) - \psi(x)]  \\
	&\quad +\bbe[V(x_{0},x) - V(x_T,x)].
	\end{align*}
	Now using the facts that $x^s = x_{T}$, $x_0=x^{s-1}$, $\tilde{x}^s =\tsum_{t=1}^{T_s}(\theta_t \bar x_t)/\tsum_{t=1}^{T_s} \theta_t$, $\tilde x = \tilde x^{s-1}$,
	and the convexity of $\psi$, we have
	\begin{align*}
	\tsum_{t=1}^{T_s} \theta_t  \bbe[\psi(\tilde{x}^s) - \psi(x)]
	& \le  \left[\tfrac{\gamma_s}{\alpha_s} (1-\alpha_s) + (T_s-1) \tfrac{\gamma_s p_s}{\alpha_s}\right] \bbe[\psi(\tilde{x}^{s-1}) - \psi(x)]  \\
	&\quad +\bbe[V(x^{s-1},x) - V(x^s,x)],
	\end{align*}
	which, in view of the fact that
	$
	\tsum_{t=1}^{T_s} \theta_t =  \tfrac{\gamma_s}{\alpha_s} + (T_s-1)  \tfrac{\gamma_s (\alpha_s + p_s)}{\alpha_s},
	$
	then implies that
	\begin{align}\label{eq:smooth_iter_reuse}
	{\cal L}_s \bbe[\psi(\tilde{x}^s) - \psi(x)] \le {\cal R}_s \bbe[\psi(\tilde{x}^{s-1}) - \psi(x)]
	+\bbe[V(x^{s-1},x) - V(x^s,x)].
	\end{align}
	Summing over the above relations, using the convexity of $\psi$ and rearranging
	the terms, we then obtain \eqnok{VRASGD_smooth}.
\end{proof}

\vgap

With the help of Lemma \ref{lem:deter_smooth_one_epoch}, we are now ready to prove Theorem \ref{Them:main-deter-smooth},
which shows that for solving smooth convex problems the \algone algorithm can achieve a fast linear rate of convergence ${\cal O}\{m\log\frac{D_0}{\epsilon}\}$
if $m \geq D_0/\epsilon$ 
and an optimal sublinear rate of convergence otherwise.

\begin{proofof}{Theorem \ref{Them:main-deter-smooth}}
	Let the probabilities $q_i=L_i / \tsum_{i=1}^m L_i$ for $i = 1, \ldots, m$, and $\theta_t$, $\gamma_s$, $p_s$, $T_s$ and $\alpha_s$ be defined as in \eqnok{def_theta_acc_SVRG}, \eqref{parameter-deter-smooth1} and \eqref{parameter-deter-alpha-sm}.
	By the definition of $L_Q$ in \eqnok{def_LQ} and the selection of $q_i$, we have $L_Q = L$.
	Observe that both conditions in \eqnok{VRASGD_cond1} and \eqnok{VRASGD_cond2} are satisfied since
	\[
	1+\mu \gamma_s - L \alpha_s \gamma_s = 1 - L \alpha_s \gamma_s = \tfrac{2}{3}
	\]
	and
	\begin{align*}
	p_s -  \tfrac{L_Q \alpha_s \gamma_s}{1+\mu \gamma_s - L \alpha_s \gamma_s} =
	p_s - \tfrac{1}{2} =0.
	\end{align*}
	Now letting ${\cal L}_s$ and ${\cal R}_s$ be defined in \eqnok{def_LHs}, we will show  that ${\cal L}_s \ge {\cal R}_{s+1}$
	for any $s \ge 1$. Indeed, if $1 \le s < s_0$, we have $\alpha_{s+1} = \alpha_s$,
	$\gamma_{s+1} = \gamma_s$, $T_{s+1} = 2 T_s$, and hence
	\begin{align*}
	w_s &= {\cal L}_s - {\cal R}_{s+1} = \tfrac{\gamma_s}{\alpha_s}
	\left[ 1 + (T_s -1)   (\alpha_s + p_s)
	- (1-\alpha_s) - (2 T_{s}-1)  p_{s} \right] \\
	&=\tfrac{\gamma_s}{\alpha_s} \left[  T_s (\alpha_s - p_s )\right] =  0.
	\end{align*}
	Moreover, if $ s \ge s_0$, we have
	\begin{align*}
	w_s &= {\cal L}_s - {\cal R}_{s+1} = \tfrac{\gamma_s}{\alpha_s} - \tfrac{\gamma_{s+1}}{\alpha_{s+1}} (1-\alpha_{s+1})
	+ (T_{s_0} -1) \left[ \tfrac{\gamma_s (\alpha_s + p_s)}{\alpha_s}
	- \tfrac{\gamma_{s+1} p_{s+1}}{\alpha_{s+1}}\right]\\
	&= \tfrac{1}{12L} + \tfrac{(T_{s_0} -1)[2 (s-s_0+4) - 1]}{24 L} \ge 0.
	\end{align*}
	Using these observations in \eqnok{VRASGD_smooth} iteratively, we then conclude that
	\begin{align*}
	{\cal L}_s \bbe[\psi(\tilde{x}^s) - \psi(x)] &\le {\cal R}_1 \bbe[\psi(\tilde{x}^{0}) - \psi(x)]
	+\bbe[V(x^{0},x) - V(x^s,x)]\\
	&\le \tfrac{2}{3L} [\psi(x^{0}) - \psi(x)] + V(x^{0},x)
	\end{align*}
	for any $s \ge 1$, where the last identity follows from the fact that ${\cal R}_1 = \tfrac{2}{3L}$.
	Recalling that $D_0:= 2[\psi(x^0) - \psi(x)] + 3L V(x^0, x)$ in \eqref{def_D_0}, now we distinguish the following two cases.
	
	{\bf Case 1:} if $s \le s_0$,
	$
	{\cal L}_s = \tfrac{2^{s+1}}{3L}.
	$
	Therefore, we have
	\[
	\bbe[\psi(\tilde{x}^s) - \psi(x)] \le 2^{-(s+1)} D_0, \quad 1\le s \le s_0.
	\]
	
	{\bf Case 2:} if $s \ge s_0$, we have
	\begin{align}
	{\cal L}_s &=
	\tfrac{1}{3L \alpha_s^2}\left[ 1+ (T_s -1) (\alpha_s + \tfrac{1}{2} )\right] \notag\\
	&=  \tfrac{(s-s_0+4) (T_{s_0} -1)}{6L} + \tfrac{(s-s_0+4)^2(T_{s_0} +1)}{24 L} \notag\\
	&\ge \tfrac{(s-s_0+4)^2 m}{48 L}, \label{eq:smooth_L_reuse}
	\end{align}
	where the last inequality follows from $T_{s_0} = 2^{\lfloor \log_2 m\rfloor +1-1} \ge m/2$.
	Hence, we obtain
	\[
	\bbe[\psi(\tilde{x}^s) - \psi(x)] \le
	\tfrac{16 D_0}{(s-s_0+4)^2 m} ,\quad  s > s_0.
	\]
	In conclusion, we have for any $x \in X$,
	\beq \label{results:acc_SVRG_non_stronglyconvex}
	\bbe[\psi(\tilde{x}^s) - \psi(x)] \le
	\begin{cases}
		2^{-(s+1)} D_0,  & 1\le s \le s_0,\\
		\tfrac{16 D_0}{(s-s_0+4)^2 m} , & s > s_0.
	\end{cases}
	\eeq
	In order to derive the complexity bounds in Theorem~\ref{Them:main-deter-smooth},
	let us first consider the region of relatively low accuracy and/or large number of components, i.e., $m \ge D_0/\epsilon$.
	In this case \algone needs to run at most $s_0$ epochs because by the first case of \eqref{results:acc_SVRG_non_stronglyconvex} we can easily check that
	\[
	\tfrac{D_0}{2^{s_0+1}} \le \epsilon.
	\]
	More precisely, the number of epochs can be bounded by
	$
	S_l := \min\left\{  \log \tfrac{D_0}{\epsilon}, s_0 \right\}.
	$
	Hence the total number of gradient evaluations can be bounded
	by
	\begin{align}
	m S_l + \tsum_{s=1}^{S_l} T_s &= m S_l + \tsum_{s=1}^{S_l} 2^{s-1}
	= {\cal O}\left\{  \min\left(m \log \tfrac{D_0}{\epsilon}, m \log m \right) \right\}
	= {\cal O}\left\{  m \log \tfrac{D_0}{\epsilon} \right\}, \label{eq:smooth_sfo_reuse1}
	\end{align}
	where the last identity follows from the assumption that $m \ge D_0/\epsilon $.
	Now let us consider the region for high accuracy and/or smaller number of components, i.e., $m < D_0/\epsilon$. In this case, we may need to run the algorithm
	for more than $s_0$ epochs. More precisely, the total number of epochs can be bounded by
	$
	S_h :=\left \lceil \sqrt{\tfrac{16 D_0}{m \epsilon}} + s_0 -4 \right\rceil.
	$
	Note that the total number of gradient evaluations needed for the first $s_0$ epochs
	can be bounded by $m s_0 + \tsum_{s=1}^{s_0} T_s$ while the total number of gradient evaluations for the remaining epochs
	can be bounded by $(T_{s_0} + m)(S_h - s_0)$.
	As a consequence, the total number of gradient evaluations of $f_i$ can be bounded by
	\begin{align}
	m s_0 + \tsum_{s=1}^{s_0} T_s + (T_{s_0} + m)(S_h - s_0)
	&\leq \tsum_{s=1}^{s_0} T_s +  (T_{s_0} + m) S_h
	= {\cal O} \left\{\sqrt{\tfrac{m D_0}{\epsilon}} + m \log m \right\}.\label{eq:smooth_sfo_reuse2}
	\end{align}
	Therefore, the results of Theorem \ref{Them:main-deter-smooth} follows immediately by combining these two cases.
\end{proofof}

\vgap

\subsection{Convex finite-sum problems with or without strong convexity}\label{sec:deter-unified-pf}
In this subsection, we provide a unified analysis of \algone when $f$ is possibly strongly convex, i.e., $\mu \ge 0$ in \eqref{def_strongconvexity}.
In particular, it achieves a stronger rate of convergence than other RIG methods if the condition number $L/\mu$ is very large.
Below we consider four different cases and establish the convergence properties of \algone in each case.

\begin{lemma}
	If $s \le s_0$, then for any $x \in X$,
	\[
	\bbe[\psi(\tilde{x}^s) - \psi(x)] \le 2^{-(s+1)} D_0, \quad 1\le s \le s_0,
	\]
	where $D_0$ is defined in \eqref{def_D_0}.
\end{lemma}

\begin{proof}
	In this case, we have $\alpha_s = p_s = \tfrac{1}{2}$, $\gamma_s = \tfrac{2}{3L}$, and $T_s = 2^{s-1}$.
	It then follows from \eqref{eq:lemma_iter} that
	\begin{align*}
	\tfrac{\gamma_s}{\alpha_s}  \bbe[\psi(\bar x_t) - \psi(x)]
	+ (1 + \mu \gamma_s) \bbe[V(x_t,x)] &\le  \tfrac{\gamma_s }{2\alpha_s}  \bbe[\psi(\tilde x) - \psi(x)] +\bbe[V(x_{t-1},x)].
	\end{align*}
	Summing up the above relation from $t = 1$ to $T_s$, we have
	\begin{align*}
	\tfrac{\gamma_s}{\alpha_s} \tsum_{t=1}^{T_s} \bbe[\psi(\bar x_t) - \psi(x)] + \bbe[V(x_{T_s},x)] +  \mu \gamma_s \tsum_{t=1}^{T_s} \bbe[V(x_t,x)]\\
	\quad \quad \le \tfrac{\gamma_sT_s}{2 \alpha_s}   \bbe[\psi(\tilde x) - \psi(x)] + \bbe[V(x_0, x)].
	\end{align*}
	Note that in this case $\theta_t$ are chosen as in \eqref{def_theta_acc_SVRG},
	i.e., $\theta_t = \tfrac{\gamma_s}{\alpha_s}$, $t = 1, \ldots, T_s$ in the definition of $\tilde x^s$, we then have
	\begin{align*}
	\tfrac{4 T_s}{3L}  \bbe[\psi(\tilde x^s) - \psi(x)] +  \bbe[V(x^s, x)]
	&\le \tfrac{4T_s}{6L}   \bbe[\psi(\tilde x^{s-1}) - \psi(x)] + \bbe[V(x^{s-1}, x)]\\
	&= \tfrac{4T_{s-1}}{3L}   \bbe[\psi(\tilde x^{s-1}) - \psi(x)] + \bbe[V(x^{s-1}, x)],
	\end{align*}
	where we use the facts that $\tilde x = \tilde x^{s-1}$, $x_0 = x^{s-1}$, and $x^s = x_{T_s}$ in the epoch $s$ and the parameter settings in \eqref{parameter-deter-smooth1}.
	Applying this inequality recursively, we then have
	\begin{align}
	\tfrac{4 T_s}{3L}  \bbe[\psi(\tilde x^s) - \psi(x)] + \bbe[V(x^s, x)]
	&\le \tfrac{2}{3L}   \bbe[\psi(\tilde x^{0}) - \psi(x)] + V(x^{0}, x) \notag\\
	&= \tfrac{2}{3L}   \bbe[\psi(x^{0}) - \psi(x)] + V(x^{0}, x).\label{eq:key2}
	\end{align}
	By plugging $T_s=2^{s-1}$ into the above inequality, we obtain the result.
\end{proof}

\vgap

\begin{lemma}
	If $s\ge s_0$ and $m\ge \tfrac{3L}{4\mu}$,
	\begin{align*}
	\bbe[\psi(\tilde x^s) - \psi(x^*)]
	&\le  \left(\tfrac{4}{5}\right)^s D_0,
	\end{align*}
	where $x^*$ is an optimal solution of \eqref{cp}.
\end{lemma}
\begin{proof}
	In this case,
	we have $\alpha_s = p_s = \tfrac{1}{2}$, $\gamma_s = \gamma = \tfrac{2}{3L}$, and $T_s  \equiv T_{s_0} = 2^{s_0-1}, s\ge s_0$.
	It then follows from \eqref{eq:lemma_iter} that
	\begin{align*}
	\tfrac{4}{3L}  \bbe[\psi(\bar x_t) - \psi(x)] + (1 + \tfrac{2\mu}{3L}) \bbe[V(x_t,x)]
	&\leq \tfrac{2}{3L}  \bbe[\psi(\tilde x) - \psi(x)] +\bbe[V(x_{t-1},x)].
	\end{align*}
	Multiplying both sides of the above inequality by $\Gamma_{t-1}=(1 + \tfrac{2 \mu}{3 L})^{t-1}$, we obtain
	\begin{align*}
	\tfrac{4}{3L} \Gamma_{t-1} \bbe[\psi(\bar x_t) - \psi(x)] +\Gamma_{t} \bbe[V(x_t,x)] \leq \tfrac{2}{3L} \Gamma_{t-1} \bbe[\psi(\tilde x) - \psi(x)] +  \Gamma_{t-1} \bbe[V(x_{t-1},x)].
	\end{align*}
	Note that $\theta_t$ are chosen as in \eqref{def_theta_acc_SVRG_sc} when $s\ge s_0$, i.e.,
	$\theta_t=\Gamma_{t-1}=(1 + \tfrac{2 \mu}{3 L})^{t-1}$, $t = 1, \ldots, T_s$, $s\ge s_0$.
	Summing up the above inequality
	for $t = 1, \ldots, T_s$ we have
	\begin{align*}
	&\tfrac{4}{3L} \tsum_{t=1}^{T_s} \theta_t  \bbe[\psi(\bar x_t) - \psi(x)] +\Gamma_{T_s} \bbe[V(x_{T_s},x)] \\
	&\quad \quad \quad \leq \tfrac{2}{3L} \tsum_{t=1}^{T_s} \theta_t  \bbe[\psi(\tilde x) - \psi(x)] +   \bbe[V(x_{0},x)], \  s\ge s_0.
	\end{align*}
	Observe that for $s\ge s_0$, $m \ge T_s \equiv T_{s_0} = 2^{\lfloor \log_2 m \rfloor} \ge m /2$, and hence that
	\begin{align}
	\Gamma_{T_s}=(1 + \tfrac{2\mu}{3L})^{T_s}= (1 + \tfrac{2\mu}{3L})^{T_{s_0}} \ge 1 + \tfrac{2\mu T_{s_0} }{3L} \ge 1 + \tfrac{T_{s_0}}{2m}\ge \tfrac{5}{4}, \ \forall s\ge s_0, \label{eq:origin}
	\end{align}
	and using the facts that $\tilde{x}^s =\tsum_{t=1}^{T_s}(\theta_t \bar x_t)/\tsum_{t=1}^{T_s} \theta_t$, $\tilde x = \tilde x^{s-1}$, $x_0 = x^{s-1}$, and $x_{T_s}=x^s $ in the $s$ epoch,
	and $\psi(\tilde x^s) - \psi(x^*) \ge 0$,
	we conclude from the above inequalities that
	\begin{align*}
	&\tfrac{5}{4} \left\{ \tfrac{2}{3L}  \bbe[\psi(\tilde x^s) - \psi(x^*)] + ( \tsum_{t=1}^{T_s} \theta_t)^{-1} \bbe[V(x^s,x^*)]\right\}\\
	& \quad \quad \le\tfrac{2}{3L} \bbe[\psi(\tilde x^{s-1}) - \psi(x^*)] + ( \tsum_{t=1}^{T_s} \theta_t)^{-1} \bbe[V(x^{s-1},x^*)], s\ge s_0.
	\end{align*}
	Applying this relation recursively for $s\ge s_0$, we then obtain
	\begin{align*}
	&\tfrac{2}{3L}  \bbe[\psi(\tilde x^s) - \psi(x^*)] + ( \tsum_{t=1}^{T_s} \theta_t)^{-1} \bbe[V(x^s,x^*)] \\
	&\le \left(\tfrac{4}{5} \right)^{s-s_0}  \left\{\tfrac{2}{3L} \bbe[\psi(\tilde x^{s_0}) - \psi(x^*)] + ( \tsum_{t=1}^{T_s} \theta_t)^{-1} \bbe[V(x^{s_0},x^*)]\right\}\\
	&\le  \left(\tfrac{4}{5}\right)^{s-s_0}  \left\{\tfrac{2}{3L} \bbe[\psi(\tilde x^{s_0}) - \psi(x^*)] + \tfrac{1}{T_{s_0}}\bbe[V(x^{s_0},x^*)]\right\},
	\end{align*}
	where the last inequality follows from
	$\tsum_{t=1}^{T_s}\theta_t \ge T_s=T_{s_0}$.
	Plugging \eqref{eq:key2} into the above inequality, we have
	\begin{align*}
	\bbe[\psi(\tilde x^s) - \psi(x^*)]
	&\le  \left(\tfrac{4}{5}\right)^{s-s_0} \tfrac{D_0}{2T_{s_0}}
	=  \left(\tfrac{4}{5}\right)^{s-s_0} \tfrac{D_0}{2^{s_0}}
	\le  \left(\tfrac{4}{5}\right)^s D_0, \ s\ge s_0.
	\end{align*}
	
\end{proof}

\vgap

\begin{lemma}
	If $s_0 <s \le s_0+\sqrt{\tfrac{12L}{m \mu}}-4$ and $m < \tfrac{3L}{4\mu}$,
	then for any $x \in X$,
	\begin{align*}
	\bbe[\psi(\tilde x^s) - \psi(x)]
	\le  \tfrac{16 D_0}{(s-s_0+4)^2 m}.
	\end{align*}
\end{lemma}

\begin{proof}
	In this case, $\tfrac{1}{2} \ge\tfrac{2}{s-s_0+4}\ge \sqrt{\tfrac{m \mu}{3L}}$.
	Therefore, we set $\theta_t$ as in \eqref{def_theta_acc_SVRG},  $\alpha_s=\tfrac{2}{s-s_0+4}, p_s = \tfrac{1}{2}$, $\gamma_s = \tfrac{1}{3L\alpha_s}$, and $T_s \equiv T_{s_0}$.
	Observe that
	the parameter setting in this case is the same as the smooth case in Theorem~\ref{Them:main-deter-smooth}.
	Hence, by following the same procedure as in the proof of Theorem \ref{Them:main-deter-smooth}, we can obtain
	\begin{align}\label{case3-reuse}
	{\cal L}_s \bbe[\psi(\tilde{x}^s) - \psi(x)] +\bbe[V(x^s,x)]
	&\le {\cal R}_{s_0+1} \bbe[\psi(\tilde{x}^{s_0}) - \psi(x)]
	+\bbe[V(x^{s_0},x)] \nn\\
	&\le {\cal L}_{s_0}\bbe[\psi(\tilde{x}^{s_0}) - \psi(x)]
	+\bbe[V(x^{s_0},x)]\nn\\
	& \le \tfrac{D_0}{3L},
	\end{align}
	where the last inequality follows from the fact that ${\cal L}_{s_0}\ge \tfrac{2T_{s_0}}{3L}$ and
	the relation in \eqref{eq:key2}. The result then follows by
	noting that ${\cal L}_s \ge \tfrac{(s-s_0+4)^2 m}{48 L}$ (see \eqref{eq:smooth_L_reuse}).
\end{proof}

\vgap

\begin{lemma}
	If  $s > \bar s_0 :=s_0+\sqrt{\tfrac{12L}{m \mu}}-4$ and $m < \tfrac{3L}{4\mu}$,
	then
	\begin{align}
	\bbe[\psi(\tilde x^s) - \psi(x^*)]
	&\leq \Big(1+\sqrt{\tfrac{\mu}{3mL}}\Big)^{\tfrac{-m(s-\bar s_0)}{2}}\tfrac{D_0}{3L/4\mu}, \label{eq:case4_unified}
	\end{align}
	where $x^*$ is an optimal solution of \eqref{cp}.
\end{lemma}

\begin{proof}
	In this case, $\tfrac{1}{2}\ge \sqrt{\tfrac{m \mu}{3L}}\ge\tfrac{2}{s-s_0+4}$.
	Therefore, we use constant step-size policy that $\alpha_s\equiv\sqrt{\tfrac{m \mu}{3L}}, p_s \equiv \tfrac{1}{2}$, $\gamma_s \equiv \tfrac{1}{3L\alpha_s}=\tfrac{1}{\sqrt{ 3m L \mu}}$, and $T_s \equiv T_{s_0}$.
	Also note that in this case $\theta_t$ are chosen as in \eqref{def_theta_acc_SVRG_sc}.
	Multiplying both sides of \eqnok{eq:lemma_iter} by $\Gamma_{t-1}=(1+\mu \gamma_s)^{t-1}$, we obtain
	\begin{align*}
	& \tfrac{\gamma_s}{\alpha_s} \Gamma_{t-1} \bbe[\psi(\bar x_t) - \psi(x)] +\Gamma_{t} \bbe[V(x_t,x)] \le  \tfrac{\Gamma_{t-1}\gamma_s}{\alpha_s} (1-\alpha_s - p_s) \bbe[\psi(\bar x_{t-1}) - \psi(x)] \notag \\
	&\quad+ \tfrac{\Gamma_{t-1}\gamma_s p_s}{\alpha_s} \bbe[\psi(\tilde x) - \psi(x)] +  \Gamma_{t-1} \bbe[V(x_{t-1},x)].
	\end{align*}
	Summing up the above inequality from $t = 1, \ldots, T_s$ and using the fact that $\bar x_0 = \tilde x$, we arrive at
	\begin{align*}
	& \tfrac{\gamma_s}{\alpha_s} \tsum_{t=1}^{T_s} \theta_t  \bbe[\psi(\bar x_t) - \psi(x)]
	+\Gamma_{T_s} \bbe[V(x_{T_s},x)]\\
	& \quad \le \tfrac{\gamma_s}{\alpha_s}  \left[ 1-\alpha_s - p_s +  p_s\tsum_{t=1}^{T_s} \Gamma_{t-1}\right] \bbe[\psi(\tilde x) - \psi(x)] + \bbe[V(x_{0},x)].
	\end{align*}
	Now using the facts that $x^s = x_{T_s}$, $x_0=x^{s-1}$, $\tilde{x}^s =\tsum_{t=1}^{T_s}(\theta_t \bar x_t)/\tsum_{t=1}^{T_s} \theta_t$, $\tilde x = \tilde x^{s-1}$, $T_s = T_{s_0}$
	and the convexity of $\psi$, we obtain
	\begin{align}
	&\tfrac{\gamma_s}{\alpha_s} \tsum_{t=1}^{T_{s_0}} \theta_t  \bbe[\psi(\tilde x^s) - \psi(x)]
	+\Gamma_{T_{s_0}}\bbe[V(x^s,x)] \notag \\
	& \quad \le \tfrac{\gamma_s}{\alpha_s}  \left[ 1-\alpha_s - p_s +  p_s\tsum_{t=1}^{T_{s_0}} \Gamma_{t-1}\right] \bbe[\psi(\tilde x^{s-1}) - \psi(x)] + \bbe[V(x^{s-1},x)] \label{eq:strong_convex_iter_key}
	\end{align}
	for any $s>\bar s_0$.
	Moreover, we have
	\begin{align*}
	\tsum_{t=1}^{T_{s_0}} \theta_t &= \Gamma_{T_{s_0}-1} + \tsum_{t=1}^{T_{s_0}-1} (\Gamma_{t-1} - (1 - \alpha_s - p_s) \Gamma_{t})\\
	&= \Gamma_{T_{s_0}}(1-\alpha_s - p_s) + \tsum_{t=1}^{T_{s_0}} (\Gamma_{t-1} - (1 - \alpha_s - p_s) \Gamma_{t})\\
	&= \Gamma_{T_{s_0}} (1-\alpha_s - p_s) + [1- (1-\alpha_s - p_s) (1+\mu \gamma_s)] \tsum_{t=1}^{T_{s_0}} \Gamma_{t-1}.
	\end{align*}
	Observe that for any $T > 1$ and $0 \le \delta T\le 1$,
	$(1+\delta)^T
	\le 1 + 2T \delta$,
	$\alpha_s = \sqrt{\tfrac{m\mu}{3L}}\ge \sqrt{\tfrac{T_{s_0}\mu}{3L}}$
	and hence that
	\begin{align*}
	1- (1-\alpha_s - p_s) (1+\mu \gamma_s)
	&\ge (1+\mu \gamma_s)(\alpha_s-\mu \gamma_s+ p_s) \\
	&\ge (1+\mu \gamma_s)(T_{s_0}\mu \gamma_s-\mu \gamma_s+ p_s) \\
	&= p_s (1+\mu \gamma_s) [2 (T_{s_0}-1) \mu \gamma_s + 1]\\
	&\ge p_s (1+\mu \gamma_s)^{T_{s_0}} = p_s \Gamma_{T_{s_0}}.
	\end{align*}
	Then we conclude that
	$
	\tsum_{t=1}^{T_{s_0}} \theta_t \ge \Gamma_{T_{s_0}} \left[ 1-\alpha_s - p_s + p_s  \tsum_{t=1}^{T_{s_0}} \Gamma_{t-1}\right].
	$
	Together with \eqref{eq:strong_convex_iter_key} and the fact that $\psi(\tilde x^s) - \psi(x^*) \ge 0$, we have
	\begin{align*}
	&\Gamma_{T_{s_0}} \left\{\tfrac{\gamma_s}{\alpha_s}  \left[ 1-\alpha_s - p_s +  p_s\tsum_{t=1}^{T_{s_0}} \Gamma_{t-1}\right] \bbe[\psi(\tilde x^s) - \psi(x^*)]
	+\bbe[V(x^s,x^*)] \right\} \\
	& \quad \le \tfrac{\gamma_s}{\alpha_s}  \left[ 1-\alpha_s - p_s +  p_s\tsum_{t=1}^{T_{s_
			0}} \Gamma_{t-1}\right] \bbe[\psi(\tilde x^{s-1}) - \psi(x^*)] + \bbe[V(x^{s-1},x^*)].
	\end{align*}
	Applying the above relation recursively for $s>\bar s_0=s_0+\sqrt{\tfrac{12L}{m \mu}}-4$, and also noting that $\Gamma_t= (1+\mu \gamma_s)^{t}$ and
	the constant step-size policy in this case, we obtain
	\begin{align*}
	&\tfrac{\gamma_s}{\alpha_s}  \left[ 1-\alpha_s - p_s +  p_s\tsum_{t=1}^{T_{s_0}} \Gamma_{t-1}\right] \bbe[\psi(\tilde x^s) - \psi(x^*)] +\bbe[V(x^s,x^*)]  \\
	& \le (1+\mu \gamma_s)^{-T_{s_0}(s-\bar s_0)}\left\{\tfrac{\gamma_s}{\alpha_s}  \left[ 1-\alpha_s - p_s +  p_s\tsum_{t=1}^{T_{s_0}} \Gamma_{t-1}\right] \right. \\
	& \quad \left. \bbe[\psi(\tilde x^{\bar s_0}) - \psi(x^*)] + \bbe[V(x^{\bar s_0},x^*)]\right\}.
	\end{align*}
	According to the parameter settings in this case, i.e.,  $\alpha_s\equiv\sqrt{\tfrac{m \mu}{3L}}, p_s \equiv \tfrac{1}{2}$, $\gamma_s \equiv \tfrac{1}{3L\alpha_s}=\tfrac{1}{\sqrt{ 3m L \mu}}$, and $\bar s_0=s_0+\sqrt{\tfrac{12L}{m \mu}}-4$, we have  $\tfrac{\gamma_{s}}{\alpha_{s}}  \left[ 1-\alpha_{s} - p_{s} +  p_{s}\tsum_{t=1}^{T_{s_0}} \Gamma_{t-1}\right]\geq \tfrac{\gamma_{s}p_{s}T_{s_0}}{\alpha_{s}}= \tfrac{T_{s_0}}{2m\mu}=\tfrac{(\bar s_0-s_0+4)^2T_{s_0}}{24L}$.
	Using this observation in the above inequality, we then conclude that
	\begin{align}
	& \bbe[\psi(\tilde x^s) - \psi(x^*)]
	\le (1+\mu \gamma_s)^{-T_{s_0}(s-\bar s_0)} \left[\bbe[\psi(\tilde x^{\bar s_0}) - \psi(x^*)] + \tfrac{24L}{(\bar s_0-s_0+4)^2T_{s_0}} \bbe[V(x^{\bar s_0},x^*)] \right]\nn\\
	&\le (1+\mu \gamma_s)^{-T_{s_0}(s-\bar s_0)}\tfrac{24L}{(\bar s_0-s_0+4)^2T_{s_0}}\left[\mathcal{L}_{\bar s_0}\bbe[\psi(\tilde x^{\bar s_0}) - \psi(x^*)] + \bbe[V(x^{\bar s_0},x^*)] \right] \nn \\
	&\le (1+\mu \gamma_s)^{-T_{s_0}(s-\bar s_0)}\tfrac{24L}{(\bar s_0-s_0+4)^2T_{s_0}}\tfrac{D_0}{3L}\nn\\
	&\le (1+\mu \gamma_s)^{-T_{s_0}(s-\bar s_0)}\tfrac{16D_0}{(\bar s_0-s_0+4)^2m}\nn\\
	&= (1+\mu \gamma_s)^{-T_{s_0}(s-\bar s_0)}\tfrac{D_0}{3L/4\mu}\nn,
	\end{align}
	where the second inequality follows from the fact that $\mathcal{L}_{\bar s_0} \ge \tfrac{(\bar s_0-s_0+4)^2T_{s_0}}{24L}= \tfrac{T_{s_0}}{2m\mu}$ due to \eqref{eq:smooth_L_reuse},
	the third inequality follows from \eqref{case3-reuse} in Case 3,
	and last inequality follows from $T_{s_0} = 2^{\lfloor \log_2 m \rfloor} \ge m /2$.
\end{proof}

\vgap

Putting the above four technical results together, we are ready to prove Theorem~\ref{Them:main-deter-sc} for \algone solving \eqref{cp} when \eqref{cp} is possibly strongly convex.

\begin{proofof}{Theorem \ref{Them:main-deter-sc}}
	Suppose that the probabilities $q_i$'s are set to $L_i / \tsum_{i=1}^m L_i$ for $i = 1, \ldots, m$.
	Moreover, let us denote $s_0 := \lfloor \log m\rfloor+1$ and assume that the
	weights $\{\theta_t\}$ are set to \eqref{def_theta_acc_SVRG} if $1\le s \le s_0$ or $s_0< s \le s_0+\sqrt{\tfrac{12L}{m \mu}}-4, \ m < \tfrac{3L}{4\mu}$.
	Otherwise, they are set to \eqnok{def_theta_acc_SVRG_sc}.
	If the parameters $\{T_s\}$, $\{\gamma_s\}$ and $\{p_s\}$ set  to \eqref{parameter-deter-smooth1} with $\{\alpha_s\}$
	given by \eqnok{parameter-deter-alpha-unified}, then
	we have
	
	\beq \label{results:acc_SVRG_unified}
	\bbe[\psi(\tilde{x}^s) - \psi(x^*)] \le
	\begin{cases}
		2^{-(s+1)} D_0,  & 1\le s \le s_0,\\
		\left(\tfrac{4}{5}\right)^{s}D_0 , & s>s_0, \mathrm{~and~} m \ge \tfrac{3L}{4\mu},\\
		\tfrac{16 D_0}{(s-s_0+4)^2 m} , & s_0< s \le s_0+\sqrt{\tfrac{12L}{m \mu}}-4 \mathrm{~and~} m < \tfrac{3L}{4\mu},\\
		\big(1+\sqrt{\tfrac{\mu}{3mL}}\big)^{\tfrac{-m(s-\bar s_0)}{2}}\tfrac{D_0}{3L/4\mu} , &  s_0+\sqrt{\tfrac{12L}{m \mu}}-4 =\bar s_0< s \mathrm{~and~} m< \tfrac{3L}{4\mu},
	\end{cases}
	\eeq
	where $x^*$ is an optimal solution of \eqref{cp} and $D_0$ is defined as in \eqref{def_D_0}.
	
	Now we are ready to provide the proof for the complexity results presented in Theorem~\ref{Them:main-deter-sc}.
	Firstly, it is clear that the first case and the third case corresponds to the results of the smooth case discussed in Theorem~\ref{Them:main-deter-smooth}.
	As a consequence, the total number of gradient evaluations can also be bounded by \eqref{eq:smooth_sfo_reuse1} and \eqref{eq:smooth_sfo_reuse2}, respectively.
	Secondly, for the second case of \eqref{results:acc_SVRG_unified}, it is easy to check that \algone needs to run at most $S:={\cal O}\{\log D_0/\epsilon\}$ epochs, and hence the total number of gradient evaluations can be bounded by
	\begin{align}
	m S + \tsum_{s=1}^{S} T_s
	\leq 2m S
	= {\cal O}\left\{  m \log \tfrac{D_0}{\epsilon} \right\}. \label{eq:unified_sfo_reuse2}
	\end{align}
	Finally, let us consider the last case of \eqref{results:acc_SVRG_unified}. Since \algone only needs to run at most $S'=\bar s_0 + 2\sqrt{\tfrac{3 L}{m\mu}}\log \tfrac{D_0/\epsilon}{3L/4\mu}
	$ epochs in this case, the total number of gradient evaluations can be bounded by
	\begin{align}
	\sum_{s=1}^{S'}(m+T_s)
	&=\sum_{s=1}^{s_0}(m+T_s)+ \sum_{s=s_0+1}^{\bar s_0}(m+T_{s_0}) + (m+T_{s_0})(S'-\bar s_0) \notag \\
	&\leq 2m\log m +2m(\sqrt{\tfrac{12L}{m \mu}}-4) + 4m\sqrt{\tfrac{3 L}{m\mu}}\log \tfrac{D_0/\epsilon}{3L/4\mu} \notag\\
	&= {\cal O} \left\{m\log m +\sqrt{\tfrac{m L}{\mu}}\log \tfrac{D_0/\epsilon}{3L/4\mu} \right\},\label{eq:unified_sfo_reuse4}
	\end{align}
	Therefore, the results of Theorem \ref{Them:main-deter-sc} follows immediately from the above discussion.
\end{proofof}

\subsection{Convex finite-sum optimization under error bound}\label{sec:error_bound_proof}
In this section, we consider a class of convex finite-sum optimization problems that satisfies the error bound condition described in \eqref{cond:eb}, and establish the convergence results for applying \algone to solve it.

\begin{proofof}{Theorem~\ref{Them:main-deter-error}}
	Similar to the smooth case, according to \eqref{VRASGD_smooth}, for any $x \in X$, we have
	\begin{align*}
	{\cal L}_s \bbe[\psi(\tilde{x}^s) - \psi(x)] &\le {\cal R}_1 \bbe[\psi(\tilde{x}^{0}) - \psi(x)]
	+\bbe[V(x^{0},x) - V(x^s,x)]\\
	&\le {\cal R}_1 [\psi(x^{0}) - \psi(x)] + V(x^{0},x).
	\end{align*}
	Then we use $x^*$ to replace $x$ and use the relation of \eqref{cond:eb} to obtain
	\begin{align*}
	{\cal L}_s \bbe[\psi(\tilde{x}^s) - \psi(x^*)]
	&\le {\cal R}_1 [\psi(x^{0}) - \psi(x^*)] + \tfrac{1}{u}[\psi(x)-\psi(x^*)].
	\end{align*}
	Now, we compute ${\cal L}_s$ and ${\cal R}_1$.
	According to \eqref{eq:smooth_L_reuse}, we have ${\cal L}_s \geq \tfrac{(s-s_0+4)^2(T_{s_0} +1)}{24 L}$. We have ${\cal R}_1=\tfrac{2T_1}{3L}$ by plugging the parameters $\gamma_1$, $p_1$, $\alpha_1$ and $T_1$ into \eqref{def_LHs}.
	
	Thus, we prove \eqref{results:acc_SVRG_error} as follows (recall that $s_0=4$ and $s=s_0+4\sqrt{\tfrac{L}{\bar{\mu}m}}$):
	\begin{align*}
	\bbe[\psi(\tilde{x}^s) - \psi(x^*)] &\le
	\tfrac{16T_1+24L/\bar{\mu}}{(s-s_0+4)^2T_12^{s_0-1}} [\psi({x}^0) - \psi(x^*)] \\
	&\le \tfrac{16+24L/(\bar{\mu}T_1)}{(s-s_0+4)^22^{s_0-1}} [\psi({x}^0) - \psi(x^*)]\\
	&\le \tfrac{5}{16}\tfrac{L/(\bar{\mu}T_1)}{1+L/(\bar{\mu}m)} [\psi({x}^0) - \psi(x^*)] \\
	&\le \tfrac{5}{16} [\psi({x}^0) - \psi(x^*)],
	\end{align*}
	where the last inequality follows from $T_1=\min\{m,\tfrac{L}{\bar{\mu}}\}$.
	
	Finally, we plug $k=\log\tfrac{\psi({x}^0) - \psi(x^*)}{\epsilon}, s_0=4, s=s_0+4\sqrt{\tfrac{L}{\bar{\mu}m}}$ and $T_1=\min\{m,\tfrac{L}{\bar{\mu}}\}$  to prove \eqref{results:acc_SVRG_error_sfo}:
	\beq
	\bar N :=k(\tsum_s(m+T_s))\leq k(ms+T_1 2^{s_0}(s-s_0+1))=
	{\cal O} \big(m+\sqrt{\tfrac{m L}{\bar{\mu}}}\big)\log \tfrac{\psi({x}^0) - \psi(x^*)}{\epsilon}. \nn
	\eeq
\end{proofof}
\setcounter{equation}{0}

\section{\algone for stochastic finite-sum optimization}\label{sec:alg-stochastic}
In this section, we consider the stochastic finite-sum optimization and online learning problems,
where only noisy gradient information of $f_i$ can be accessed via the SFO oracle,
and provide the proof of Theorem~\ref{Them:main-sto-smooth}.


Before proving Theorem \ref{Them:main-sto-smooth}, we need to establish some key technical results in the following lemmas.
First, we rewrite Lemma \ref{variance_reduced_betterEsti} under the stochastic setting.
Lemma~\ref{variance_reduced_betterEsti_sto} below shows that $G_t$ updated according to Algorithm~\ref{algVRASGD_sto} is an unbiased estimator of $\nabla f(\underline x_t)$ and its variance is upper bounded.

\begin{lemma}\label{variance_reduced_betterEsti_sto}
	Conditionally on~$x_1, \ldots, x_{t}$,
	\begin{align}
	\bbe [\delta_t] &=0, \label{SVRG_bnd1_sto}\\
	\bbe[\dnorm{\delta_t}^2] &\le 2 L_Q[f(\tilde x) - f(\underline x_{t}) - \langle \nabla f(\underline x_{t}), \tilde x - \underline x_{t}\rangle]
	+\tsum_{i=1}^m \tfrac{\sigma^2}{q_i m^2 b_s} +\tsum_{i=1}^m \tfrac{2\sigma^2}{q_i m^2 B_s}+\tfrac{2\sigma^2}{m B_s}, \label{SVRG_bnd2_sto}
	\end{align}
	where $\delta_t = G_t - \nabla f (\underline x_t)$ and
	$G_t = \tfrac{1}{q_{i_t} m b_s}\tsum_{k=1}^{b_s}\big(G_{i_t}(\underline x_{t},\xi_{k}^s)
	- G_{i_t}(\tilde{x})\big)
	+ \tilde{g}$
	(see Line~\ref{eqn:estgradient_sto} of Algorithm \ref{algVRASGD_sto}).
\end{lemma}

\begin{proof}
	Take the expectation with respect to $i_t$ and $[\xi]:=\{\xi_k\}_{k=1}^{b_s}$ conditionally on~$x_1, \ldots, x_{t}$, we obtain
	\begin{align*}
	&\bbe_{i_t, [\xi]} \Big[\frac{1}{m q_{i_t}b_s}\sum_{k=1}^{b_s}G_{i_t}(\underline x_{t},\xi_{k})-\frac{1}{m q_{i_t}}G_{i_t}(\tilde{x})
	+\frac{1}{m}\sum_{i=1}^{m}G_i(\tilde{x})
	-\nabla f(\underline x_{t}) \Big] \\
	&=\bbe_{i_t} \Big[\frac{1}{m q_{i_t}}\nabla f_{i_t}(\underline x_{t})-\frac{1}{m q_{i_t}}G_{i_t}(\tilde{x})
	+\frac{1}{m}\sum_{i=1}^{m}G_i(\tilde{x})
	-\nabla f(\underline x_{t})\Big]\\
	&= 0,
	\end{align*}
	where the first equality follows from \eqref{eq:unbiased}.
	
	Moreover, we have
	\begin{align}
	\bbe [\dnorm{\delta_t}^2]
	=& \bbe \Big[\big\|
	\frac{1}{m q_{i_t}b_s}\sum_{k=1}^{b_s}G_{i_t}(\underline x_{t},\xi_{k})
	-\frac{1}{m q_{i_t}}G_{i_t}(\tilde{x})
	+\frac{1}{m}\sum_{i=1}^{m}G_i(\tilde{x})
	-\nabla f(\underline x_{t})\big\|_*^2 \Big] \notag\\
	=& \bbe \Big[\big\|
	\frac{1}{m q_{i_t}}\big(\nabla f_{i_t} (\underline x_{t})-\nabla f_{i_t}(\tilde{x})\big)
	+\nabla f(\tilde{x}) - \nabla f(\underline x_{t})\big\|_*^2 \Big]\notag\\
	&\qquad +
	\bbe \Big[\big\|
	\frac{1}{m q_{i_t}b_s}\sum_{k=1}^{b_s}G_{i_t}(\underline x_{t},\xi_{k})
	-\frac{1}{m q_{i_t}}\nabla f_{i_t} (\underline x_{t})
	\big\|_*^2 \Big] \notag\\
	&\qquad +
	\bbe \Big[\big\|
	\frac{1}{m q_{i_t}}\nabla f_{i_t} (\tilde{x})
	-\frac{1}{m q_{i_t}}G_{i_t}(\tilde{x})
	+\frac{1}{m}\sum_{i=1}^{m}G_i(\tilde{x})
	-\frac{1}{m}\sum_{i=1}^{m}\nabla f_i(\tilde{x})
	\big\|_*^2 \Big] \notag\\
	\leq& \bbe \Big[\frac{1}{m^2 q_{i_t}^2}\big\|
	\nabla f_{i_t} (\underline x_{t})-\nabla f_{i_t}(\tilde{x})
	\big\|_*^2 \Big]
	+\sum_{i=1}^m \frac{\sigma^2}{q_i m^2 b_s}\notag\\
	&\qquad +
	2\bbe \Big[\big\|
	\frac{1}{m q_{i_t}}\nabla f_{i_t} (\tilde{x})
	-\frac{1}{m q_{i_t}}G_{i_t}(\tilde{x})\big\|_*^2 \Big]
	+2\bbe \Big[\big\|
	\frac{1}{m}\sum_{i=1}^{m}G_i(\tilde{x})
	-\frac{1}{m}\sum_{i=1}^{m}\nabla f_i(\tilde{x})
	\big\|_*^2 \Big] \notag\\
	\leq& \bbe \Big[\frac{1}{m^2 q_{i_t}^2}\big\|
	\nabla f_{i_t} (\underline x_{t})-\nabla f_{i_t}(\tilde{x})
	\big\|_*^2 \Big]
	+\sum_{i=1}^m \frac{\sigma^2}{q_i m^2 b_s}
	+ \sum_{i=1}^m \frac{2\sigma^2}{q_i m^2 B_s}
	+\frac{2\sigma^2}{m B_s}, \notag
	\end{align}
	where the last inequality uses \eqref{eq:bounded} and in view of relation~\eqref{SVRG_bnd0} (with $x$ and $x^*$ replaced by $\tilde x$ and $\underline x_{t}$), then implies \eqref{SVRG_bnd2_sto}.
\end{proof}

We are now ready to rewrite Lemma~\ref{Lemma_VARSGD_iter} under the stochastic setting.
\begin{lemma} \label{Lemma_VARSGD_iter_sto}
	Assume that $\alpha_s \in [0,1]$, $p_s \in [0,1]$ and $\gamma_s > 0$ satisfy \eqref{VRASGD_cond1} and \eqref{VRASGD_cond2}.
	Then, conditional on $x_1, \ldots, x_{t-1}$, we have
	\begin{align} \label{eq:lemma_iter_sto}
	\bbe[\psi(\bar x_t)+ \tfrac{\alpha_s(1+\mu \gamma_s)}{\gamma_s}V(x_t,x)] &\le (1-\alpha_s - p_s) \psi(\bar x_{t-1}) + p_s \psi(\tilde x) + \alpha_s \psi(x)
	+ \tfrac{\alpha_s}{\gamma_s}V(x_{t-1},x) \notag\\
	&\qquad +\tfrac{\alpha_s \gamma_s}{2(1+\mu \gamma_s - L \alpha_s \gamma_s)}
	\big(\tsum_{i=1}^m \tfrac{\sigma^2}{q_i m^2 b_s}
	+ \tsum_{i=1}^m \tfrac{2\sigma^2}{q_i m^2 B_s}
	+\tfrac{2\sigma^2}{m B_s}\big)
	\end{align}
	for any $x \in X$.
\end{lemma}
\begin{proof}
	Similar to the proof of Lemma~\ref{Lemma_VARSGD_iter}, in view of the smoothness and (strong) convexity of $f$, we recall the result in \eqref{rel1}, i.e.,
	\begin{align}
	f(\bar x_t)
	&\le  (1-\alpha_s - p_s) f(\bar x_{t-1})
	+ \alpha_s \left[\psi(x) - h(x_t) + \tfrac{1}{\gamma_s} V(x_{t-1},x) - \tfrac{1+\mu \gamma_s}{\gamma_s} V(x_t,x) \right]\nn \\
	&\quad +  p_s l_f(\underline x_t, \tilde x) + \tfrac{\alpha_s \gamma_s\|\delta_t\|_*^2}{2(1+\mu \gamma_s - L \alpha_s \gamma_s)}
	- \alpha_s \langle \delta_t,x_{t-1}^+ - x\rangle. \label{rel2}
	\end{align}
	Also note that by \eqnok{SVRG_bnd1_sto}, \eqnok{SVRG_bnd2_sto}, \eqnok{VRASGD_cond2} and the convexity of $f$, we have, conditional on $x_1, \ldots, x_{t-1}$,
	\begin{align*}
	& p_s l_f(\underline x_t, \tilde x) + \tfrac{\alpha_s \gamma_s\bbe[\|\delta_t\|_*^2]}{2(1+\mu \gamma_s - L \alpha_s \gamma_s)} + \alpha_s \bbe[\langle \delta_t,x_{t-1}^+ - x\rangle]\\
	& \le  p_s l_f(\underline x_t, \tilde x) + \tfrac{L_Q\alpha_s \gamma_s}{1+\mu \gamma_s - L \alpha_s \gamma_s} [f(\tilde x) - l_f(\underline x_t, \tilde x) ]\\
	&\quad +\tfrac{\alpha_s \gamma_s}{2(1+\mu \gamma_s - L \alpha_s \gamma_s)}\big(\tsum_{i=1}^m \tfrac{\sigma^2}{q_i m^2 b_s}
	+ \tsum_{i=1}^m \tfrac{2\sigma^2}{q_i m^2 B_s}
	+\tfrac{2\sigma^2}{m B_s}\big)\\
	&\le \left(p_s -  \tfrac{L_Q \alpha_s \gamma_s}{1+\mu \gamma_s - L \alpha_s \gamma_s}\right) l_f(\underline x_t, \tilde x)
	+  \tfrac{L_Q \alpha_s \gamma_s}{1+\mu \gamma_s - L \alpha_s \gamma_s}f(\tilde x)\\
	&\quad +\tfrac{\alpha_s \gamma_s}{2(1+\mu \gamma_s - L \alpha_s \gamma_s)}\big(\tsum_{i=1}^m \tfrac{\sigma^2}{q_i m^2 b_s}
	+ \tsum_{i=1}^m \tfrac{2\sigma^2}{q_i m^2 B_s}
	+\tfrac{2\sigma^2}{m B_s}\big)\\
	&\le p_s f(\tilde x)+\tfrac{\alpha_s \gamma_s}{2(1+\mu \gamma_s - L \alpha_s \gamma_s)}\big(\tsum_{i=1}^m \tfrac{\sigma^2}{q_i m^2 b_s}
	+ \tsum_{i=1}^m \tfrac{2\sigma^2}{q_i m^2 B_s}
	+\tfrac{2\sigma^2}{m B_s}\big).
	\end{align*}
	Moreover, by convexity of $h$, we have $h(\bar x_t) \le (1-\alpha_s - p_s) h(\bar x_{t-1}) + \alpha_s h(x_t) + p_s h(\tilde x)$.
	The result then follows by summing up the previous two conclusions with \eqref{rel2}.
\end{proof}

Finally, we need to rewrite the stochastic counterpart of the decrease of function value in each epoch (Lemma \ref{lem:deter_smooth_one_epoch}) in the following lemma.
\begin{lemma}\label{lem:sto_smooth_one_epoch}
	Assume that for each epoch $s$, $s \ge 1$, we have $\alpha_s$, $\gamma_s$, $p_s$ and $T_s$ such that \eqnok{VRASGD_cond1}-\eqnok{VRASGD_cond2} hold. Also, let us set
	$\theta_t$ as \eqref{def_theta_acc_SVRG}.
	Moreover, let ${\cal L}_s$, ${\cal R}_s$ and $w_s$ defined as in \eqref{def_LHs} and \eqref{def_ws} respectively.
	Then we have
	\begin{align}
	{\cal L}_s \bbe[\psi(\tilde{x}^s) - \psi(x)] &+ (\tsum_{j=1}^{s-1} w_j)
	\bbe[\psi(\bar x^s) - \psi(x)] \nn\\
	& \le {\cal R}_1 \bbe[\psi(\tilde{x}^{0}) - \psi(x)]
	+\bbe[V(x^{0},x) - V(x^s,x)] \notag \\
	&\quad +\tsum_{j=1}^{s}\tfrac{\gamma_{j}^2T_{j}}{2(1+\mu \gamma_{j} - L \alpha_{j} \gamma_{j})}\big(\tsum_{i=1}^m \tfrac{\sigma^2}{q_i m^2 b_j}
	+ \tsum_{i=1}^m \tfrac{2\sigma^2}{q_i m^2 B_j}
	+\tfrac{2\sigma^2}{m B_j}\big)  \label{VRASGD_smooth_sto}
	\end{align}
	for any $x \in X$, where $\bar x^s$ is defined as in \eqref{def_w_output}.
\end{lemma}
\begin{proof}
	Using our assumptions on $\alpha_s$, $\gamma_s$ and $p_s$,
	the fact that $\mu = 0$, and subtracting $\psi(x)$ from the concluding inequality \eqref{eq:lemma_iter_sto} of Lemma~\ref{Lemma_VARSGD_iter_sto}, we have
	\begin{align*}
	\tfrac{\gamma_s}{\alpha_s}  \bbe[\psi(\bar x_t) - \psi(x)] &\le  \tfrac{\gamma_s}{\alpha_s} (1-\alpha_s - p_s) \bbe[\psi(\bar x_{t-1}) - \psi(x)]
	+  \tfrac{\gamma_s p_s}{\alpha_s}  \bbe[\psi(\tilde x) - \psi(x)]\\
	&\qquad+\bbe[V(x_{t-1},x) - V(x_t,x)]\\
	&\qquad +\tfrac{\gamma_s^2}{2(1+\mu \gamma_s - L \alpha_s \gamma_s)}\big(\tsum_{i=1}^m \tfrac{\sigma^2}{q_i m^2 b_s}
	+ \tsum_{i=1}^m \tfrac{2\sigma^2}{q_i m^2 B_s}
	+\tfrac{2\sigma^2}{m B_s}\big).
	\end{align*}
	Hence following the same procedure as we did in proving Lemma~\ref{lem:deter_smooth_one_epoch}, we can obtain \eqref{VRASGD_smooth_sto}.
\end{proof}

With the help of Lemma \ref{lem:sto_smooth_one_epoch}, we are now ready to prove Theorem \ref{Them:main-sto-smooth}, which establishes the convergence properties of \algone for solving stochastic smooth finite-sum problems given in the form of \eqref{cp}.

\begin{proofof}{Theorem \ref{Them:main-sto-smooth}}
	Let the probabilities $q_i=L_i / \tsum_{i=1}^m L_i$ for $i = 1, \ldots, m$,
	we then have $L_Q = L$.
	Clearly by setting $\alpha_s$, $\gamma_s$, and $p_s$ in \eqref{parameter-deter-smooth1} and \eqref{parameter-deter-alpha-sm}, conditions \eqnok{VRASGD_cond1} and \eqnok{VRASGD_cond2} are satisfied.
	Moreover, similar to the deterministic case, by setting ${\cal L}_s$ and ${\cal R}_s$ as in \eqref{def_LHs}, we can show that ${\cal L}_s \ge {\cal R}_{s+1}$
	for any $s \ge 1$.
	Using these observations in \eqnok{VRASGD_smooth_sto}, we then conclude that
	\begin{align*}
	{\cal L}_s \bbe[\psi(\tilde{x}^s) - \psi(x)] &\le {\cal R}_1 \bbe[\psi(\tilde{x}^{0}) - \psi(x)]
	+\bbe[V(x^{0},x) - V(x^s,x)]\\
	&\qquad +\tsum_{j=1}^{s}\tfrac{3\gamma_{j}^2T_{j}}{4}\big(\tsum_{i=1}^m \tfrac{\sigma^2}{q_i m^2 b_j}
	+ \tsum_{i=1}^m \tfrac{2\sigma^2}{q_i m^2 B_j}
	+\tfrac{2\sigma^2}{m B_j}\big) \\
	&\le \tfrac{2}{3L} [\psi(x^{0}) - \psi(x)] + V(x^{0},x)\\
	&\qquad +\tsum_{j=1}^{s}\tfrac{T_{j}}{12L^2\alpha_{j}^2}\big(\tfrac{C\sigma^2}{b_{j}} +\tfrac{2C\sigma^2}{B_{j}}+\tfrac{2\sigma^2}{m B_j}\big) \\
	&\le \tfrac{2}{3L} [\psi(x^{0}) - \psi(x)] + V(x^{0},x)\\
	&\qquad +\tsum_{j=1}^{s}\tfrac{T_{j}}{12L^2\alpha_{j}^2}\big(\tfrac{C\sigma^2}{b_{j}} +\tfrac{4C\sigma^2}{B_{j}}\big)
	\end{align*}
	for any $s \ge 1$, where the second inequality follows from the fact that ${\cal R}_1 = \tfrac{2}{3L}, \gamma_s = \tfrac{1}{3 L \alpha_s}$, and the definition $C:=\tsum_{i=1}^m{\tfrac{1}{q_im^2}}$.
	Note that the last two terms $\tfrac{C\sigma^2}{b_{j}}$ and $\tfrac{4C\sigma^2}{B_{j}}$
	are in the same order.
	Also note that the sampling complexity (number of calls to the SFO oracle) is bounded by $\tsum_s mB_s+\tsum_s T_sb_s$ and the communication complexity (CC), if in the distribued machine learning case, is bounded by $\sum_s(m+T_s)$.
	So we can let $B_j\equiv b_j$, then these two complexity are bounded by their first term $m\tsum_s B_s$ and $mS$ respective (note that $T_s$ is always no larger than $m$).
	Concretely, we let
	\beq\label{eq:def_smooth_mj}
	B_j\equiv b_j:=
	\begin{cases}
		b_1(\tfrac{3}{2})^{j-1},  & j \leq s_0\\
		b' & j > s_0
	\end{cases}.
	\eeq
	Recalling that $D_0:= 2[\psi(x^0) - \psi(x)] + 3L V(x^0, x)$ in \eqref{def_D_0}, now we distinguish the following two cases.
	
	{\bf Case 1:} if $s \le s_0=\lfloor \log m\rfloor+1$,
	$
	{\cal L}_s =
	\tfrac{T_s}{3L \alpha_s^2}=\tfrac{2^{s+1}}{3L}.
	$
	Therefore, we have
	\begin{align*}
	\bbe[\psi(\tilde{x}^s) - \psi(x)] &\le 2^{-(s+1)} D_0 +2^{-(s+1)}\tsum_{j=1}^{s}\tfrac{2^{j-1}}{L}\big(\tfrac{C\sigma^2}{b_{j}}
	+\tfrac{4C\sigma^2}{B_{j}}\big) \\
	&\le 2^{-(s+1)} D_0 +2^{-(s+1)}\tsum_{j=1}^s\tfrac{5C\sigma^2 2^{j-1}}{L B_{j}} \\
	&\le 2^{-(s+1)} D_0 + 2^{-(s+1)}\tsum_{j=1}^s(\tfrac{4}{3})^{j-1}\tfrac{5C\sigma^2}{L b_1} \\
	&\le 2^{-(s+1)} D_0 + (\tfrac{2}{3})^{s}\tfrac{15C\sigma^2}{2L b_1} \\
	&= \tfrac{\epsilon}{2} + \tfrac{\epsilon}{2}, \quad 1\le s \le s_0.
	\end{align*}
	where the last equality holds when $s=\log \tfrac{D_0}{\epsilon}$ and $b_1= (\tfrac{2}{3})^s\frac{15C\sigma^2}{L\epsilon}$.
	
	In this case,
	\algone needs to run at most $S_l:=\min\left\{ \log \tfrac{D_0}{\epsilon}, s_0 \right\}$ epochs.
	Hence, the sampling complexity (number of calls to the SFO oracle) is bounded by
	\begin{align}
	\tsum_{s=1}^{S_l} (mB_s+T_sb_s)
	\leq 2m \tsum_{s=1}^{S_l} b_1(\tfrac{3}{2})^{s-1}
	\leq 4mb_1(\tfrac{3}{2})^{S_l}
	= {\cal O}\left\{ \tfrac{mC\sigma^2}{L\epsilon}\right\},  \label{eq:sto_smooth_sfo}
	\end{align}
	and the communication complexity (CC), if in the distributed machine learning case, is bounded by
	\begin{align}
	\tsum_{s=1}^{S_l} (m+T_s)
	\leq 2mS_l
	= {\cal O}\left\{ m\log \tfrac{D_0}{\epsilon}\right\}, \quad m \geq \tfrac{D_0}{\epsilon}. \label{eq:sto_smooth_cc}
	\end{align}
	
	\vspace{5pt}
	{\bf Case 2:} if $s \ge s_0$, ${\cal L}_s \ge \tfrac{(s-s_0+4)^2 T_{s_0}}{24 L}$.
	Therefore, we have
	\begin{align*}
	\bbe[\psi(\tilde{x}^s) - \psi(x)]
	& \le \tfrac{8 D_0}{(s-s_0+4)^2 T_{s_0}}
	+\tfrac{8}{(s-s_0+4)^2 T_{s_0}} \big(\tsum_{j=1}^{s_0}\tfrac{5C\sigma^2 2^{j-1}}{L B_{j}} +\tsum_{j=s_0+1}^s\tfrac{5C\sigma^2T_{s_0}}{4L b' \alpha_j^2}\big)\\
	& \le \tfrac{16 D_0}{(s-s_0+4)^2 m}
	+\tfrac{16}{(s-s_0+4)^2} \big(2^{-s_0}\tsum_{j=1}^{s_0}(\tfrac{4}{3})^{j-1}\tfrac{5C\sigma^2}{L b_1}
	+\tsum_{j=s_0+1}^s\tfrac{5C\sigma^2(j-s_0+4)^2}{32L b'}\big)\\
	& \le \tfrac{16 D_0}{(s-s_0+4)^2 m} + \tfrac{16}{(s-s_0+4)^2}(\tfrac{2}{3})^{s_0}\tfrac{15C\sigma^2}{L b_1}
	+ \tfrac{5C\sigma^2(s-s_0)}{2Lb'}\\
	&= \tfrac{\epsilon}{2} + \tfrac{\epsilon}{4} + \tfrac{\epsilon}{4}
	\quad  s > s_0.
	\end{align*}
	where the last equality holds when $s=s_0 + \sqrt{\tfrac{32D_0}{m\epsilon}} -4$, $b_1= (\tfrac{2}{3})^{s_0}\frac{30C\sigma^2m}{LD_0}$ and
	$b'=\tfrac{10C\sigma^2(s-s_0)}{L\epsilon}$.
	
	In this case,
	\algone needs to run at most run at most $S_h:= s_0 +\sqrt{\tfrac{32 D_0}{m \epsilon}} -4$  epochs.
	Hence, the sampling complexity (number of calls to the SFO oracle) is bounded by
	\begin{align}
	\tsum_{s=1}^{S_l} (mB_s+T_sb_s)
	&\leq 2m \tsum_{s=1}^{s_0} b_1(\tfrac{3}{2})^{s-1} + 2mb'(S_h-s_0)\nn\\
	&\leq 4mb_1(\tfrac{3}{2})^{s_0}+\tfrac{20mC\sigma^2(S_h-s_0)^2}{L\epsilon}
	= {\cal O}\left\{ \tfrac{C\sigma^2D_0}{L\epsilon^2}\right\},  \label{eq:sto_smooth_sfo2}
	\end{align}
	and the communication complexity (CC), if in the distributed machine learning case, is bounded by
	\begin{align}
	\tsum_{s=1}^{S_h} (m+T_s)
	\leq 2m(s_0+S_h-s_0)
	= {\cal O}\left\{ m\log m + \sqrt{\tfrac{m D_0}{\epsilon}} \right\}, \quad m < \tfrac{D_0}{\epsilon}. \label{eq:sto_smooth_cc2}
	\end{align}
	The result of Theorem \ref{Them:main-sto-smooth} follows immediately by combining these two cases.
\end{proofof}

\section{More numerical experiments}\label{sec:experiments}
\noindent{\bf Strongly convex problems with small strongly convex modulus.}
We consider ridge regression models with a small regularizer coefficient ($\lambda$) given in the following form,
\beq\label{p-ridge}
\min_{x\in \bbr^n}\{\psi(x):= \tfrac{1}{m}\tsum_{i=1}^{m}f_i(x) + h(x)\} \ \mbox{where } f_i(x):=\tfrac{1}{2}(a_i^Tx-b_i)^2, h(x):=\lambda \|x\|_2^2.
\eeq
Since the above problem is strongly convex, we compare the performance of \algone with those of Prox-SVRG\cite{xiao2014proximal} and Katyusha\cite{allen2016katyusha}.
As we can see from Figure~\ref{ridge-results}, \algone and Katyusha converges much faster than Prox-SVRG in terms of training loss.
Although \algone and Katyusha perform similar in terms of training loss per gradient calls, \algone may require less CPU time to perform one epoch than Katyusha.
In fact, \algone only needs to solve one proximal mapping per inner iteration while Katyusha requires to solve two for strongly convex problems.
\begin{figure}[H]
	\begin{minipage}{0.45\textwidth}
		\centering
		\includegraphics[scale = 0.15]{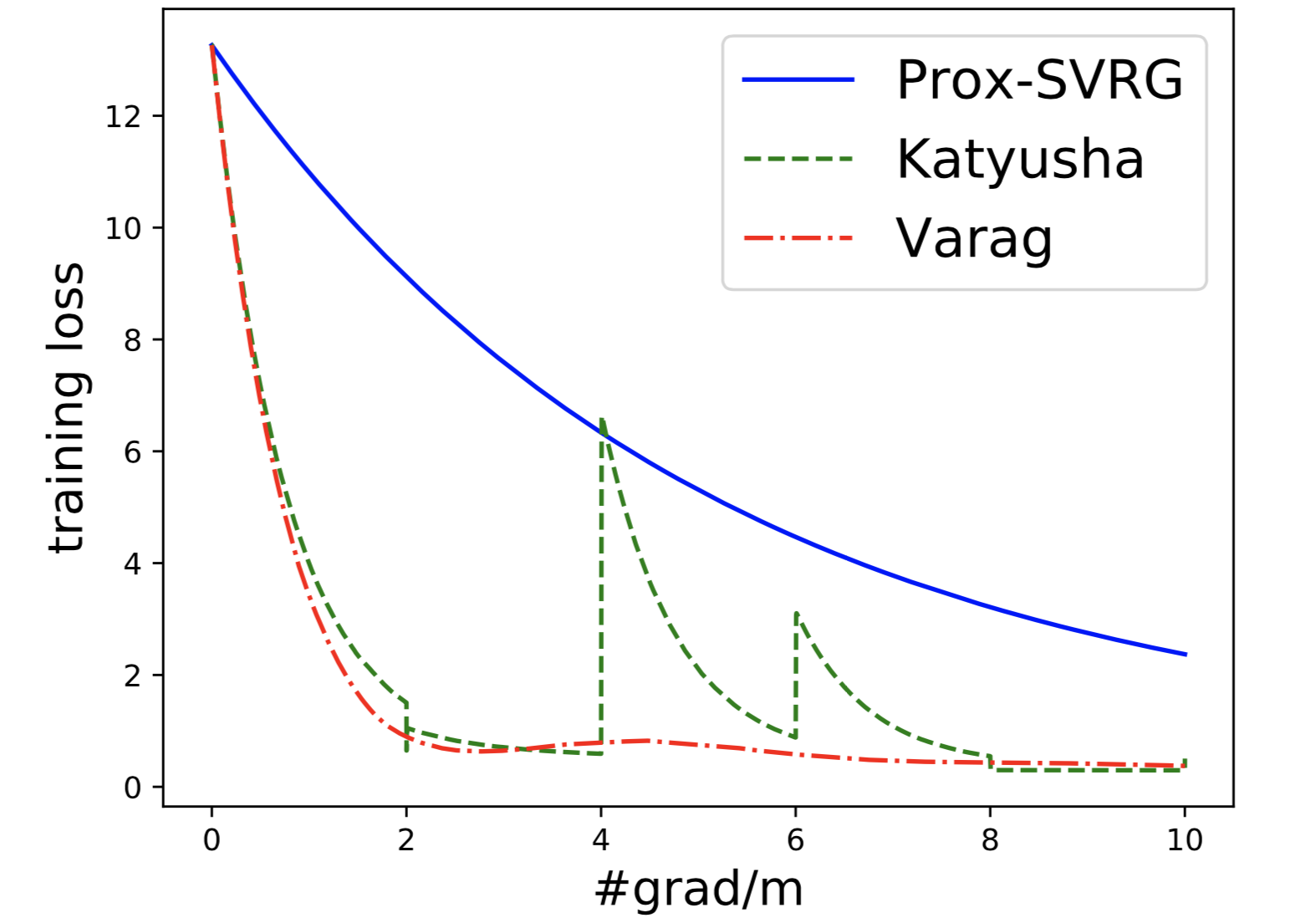}
		\\
		{\scriptsize  Diabetes ($m=1151$), ridge $\lambda=10^{-6}$}
	\end{minipage}
	\begin{minipage}{0.45\textwidth}
		\centering
		\includegraphics[scale = 0.15]{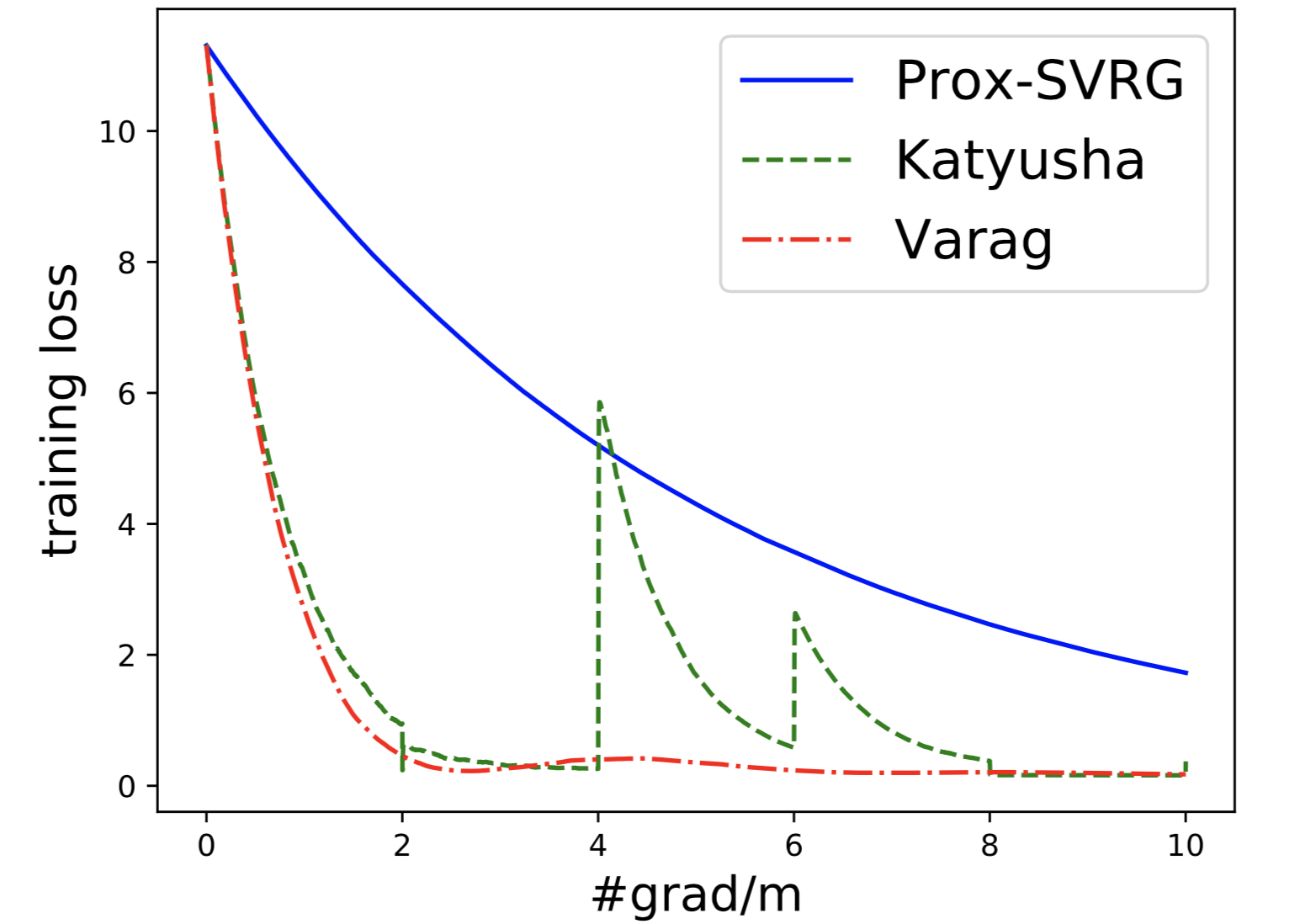}
		{\scriptsize  Breast-Cancer-Wisconsin ($m=683$), ridge $\lambda=10^{-6}$}
	\end{minipage}
	\caption{\footnotesize In this experiments, the algorithmic parameters for Prox-SVRG and Katyusha are set according to \cite{xiao2014proximal} and \cite{allen2016katyusha}, respectively, and those for \algone are set as in Theorem~\ref{Them:main-deter-sc}. }
	\label{ridge-results}
\end{figure}

\end{document}